\documentclass[graybox]{svmult}

\usepackage{helvet}         
\usepackage{courier}        
\usepackage{type1cm}        

\usepackage{mathdots}        

\usepackage{makeidx}         
\usepackage{graphicx}        
\usepackage{multicol}        
\usepackage[bottom]{footmisc}
\usepackage{bm}
\usepackage{mathtools}
\usepackage{color}
\usepackage[top=3.cm, bottom=3cm, outer=3.5cm, inner=3cm]{geometry}

\makeindex             

\usepackage{latexsym}
\usepackage{dsfont}
\usepackage{bbm}
\usepackage{amssymb}
\usepackage{amsmath}
\usepackage{color}
\usepackage{soul}

\newtheorem{Theorem}{Theorem}[section]
\newtheorem{Lemma}[Theorem]{Lemma}
\newtheorem{Cor}[Theorem]{Corollary}
\newtheorem{Prop}[Theorem]{Proposition}
\newtheorem{Rem}[Theorem]{Remark}

\def\cE{\mathcal{E}}

\def\cS{\mathcal{S}}

\def\cX{\mathcal{X}}

\def\Erw{\mathbb{E}}

\def\N{\mathbb{N}}

\def\Prob{\mathbb{P}}

\def\R{\mathbb{R}}

\def\Z{\mathbb{Z}}

\def\eps{\varepsilon}
\def\vph{\varphi}

\def\1{\vec{1}}
\def\3{{\ss}}

\def\eqdist{\stackrel{d}{=}}

\def\idist{\stackrel{d}{\to}}

\def\wh{\widehat}

\newcommand{\stirling}[2]{\genfrac{[}{]}{0pt}{}{#1}{#2}}

\newcommand{\od}{\overset{d}{=}}
\newcommand{\ofdd}{\overset{f.d.d.}{=}}
\newcommand{\dod}{\overset{d}{\to}}

\def\todistr{\stackrel{d}{\longrightarrow}}
\def\todistrfd{\stackrel{f.d.d.}{\longrightarrow}}

\def\toweak{\stackrel{w}{\longrightarrow}}

\def\uarr{\uparrow}
\def\darr{\downarrow}

\allowdisplaybreaks[4] 

\begin{document}

\title*{A leader-election procedure using records}
\titlerunning{A leader-election procedure using records}
\author{Gerold Alsmeyer$^{1}$, Zakhar Kabluchko$^{1}$ and Alexander Marynych$^{1,2}$}
\institute{$^{1}$ Inst.~Math.~Statistics, Department
of Mathematics and Computer Science, University of M\"unster,
Orl\'eans-Ring 10, D-48149 M\"unster, Germany.\at
$^{2}$ Faculty of Cybernetics, Taras Shevchenko National University of Kyiv, 01601 Kyiv, Ukraine\at
\email{gerolda@math.uni-muenster.de, zakhar.kabluchko@uni-muenster.de,\at marynych@unicyb.kiev.ua}}

\maketitle

\abstract{Motivated by the open problem of finding the asymptotic distributional behavior of the number of collisions in a Poisson-Dirichlet coalescent, the following version of a stochastic leader-election algorithm is studied. Consider an infinite family of persons, labeled by $1,2,3,\ldots$, who generate iid random numbers from an arbitrary continuous distribution. Those persons who have generated a record value, that is, a value larger than the values of all previous persons, stay in the game, all others must leave. The remaining persons are relabeled by $1,2,3,\ldots$ maintaining their order in the first round, and the election procedure is repeated independently from the past and indefinitely. We prove limit theorems for a number of relevant functionals for this procedure, notably the number of rounds $T(M)$ until all persons among $1,\ldots,M$, except the first one, have left (as $M\to\infty$). For example, we show that the sequence $(T(M)-\log^{*}M)_{M\in\N}$, where $\log^{*}$ denotes the iterated logarithm, is tight, and study its weak subsequential limits.  We further provide an appropriate and apparently new kind of normalization (based on tetrations) such that the original labels of persons who stay in the game until round $n$ converge (as $n\to\infty$) to some random non-Poissonian point process and study its properties. The results are applied to describe all subsequential distributional limits for the number of collisions in the Poisson-Dirichlet coalescent, thus providing a complete answer to the open problem mentioned above.}

\bigskip



\section{Introduction}\label{sec:intro}

Motivated by a natural connection with the number of collisions in a Poisson-Dirichlet coalescent to be explained further below (Section \ref{sec:PD}), we propose and analyze a leader-election procedure \cite{Fill+Mahmoud+Szpankowski:96,Janson+Lavault+Louchard:2008,Janson+Szpankowski:97,Kalpathy+Mahmoud+Ward:2011,Prodinger:93} which, unlike its classical version using independent, identically distributed (iid) Bernoulli trials, is based on the records in a sequence of iid continuous random variables. The procedure starts with an infinite number of players labeled by $1,2,3,\ldots$
who independently generate random numbers from an arbitrary continuous distribution
(which can be chosen w.l.o.g. to be uniform on $(0,1)$). Those players holding a record value, that is, a value larger than those of all preceding players, stay in the game for the next round. Each round is run independently from the previous ones and in the same manner after relabeling the players still in the game by $1,2,3,\ldots$ while keeping the original order. Notice that player $1$ always remains in the game and retains his number.
He can therefore be considered either as the person who is always elected as the leader, or as a ``dummy'', who does not participate in the leader-election procedure at all. Here we adopt the first interpretation so that, for instance, the time needed to select the leader is the same as the time until all players except the first one leave the game.

\vspace{.1cm}
Define the indicator variables
\begin{equation}\label{xi_definition}
\xi_{i}^{(n)}\ :=\ \1\{\text{in round }n, \text{ the player with the current number }i\text{ survives}\} 
\end{equation}
for $i,\,n\in\N$. By R\'enyi's theorem on records, see \cite[page 58]{Nevzorov:01} or \cite{Renyi:62}, the infinite random vectors $(\xi^{(n)}_{1},\xi^{(n)}_{2},\ldots)$, $n\in\N$, are independent distributional copies of $(\xi_{1},\xi_{2},\ldots)$, where $(\xi_{i})_{i\in\N}$ is a sequence of independent random variables with
\begin{equation}\label{xi_distribution}
\Prob\{\xi_{i}=1\}\ =\ 1-\Prob\{\xi_{i}=0\}=\frac{1}{i}
\end{equation}
for $i\in\N$. We are interested in the following quantities:
\begin{itemize}\itemsep2pt
\item $N^{(n)}_{M}$, the number of players among $1,2,\ldots,M$ who survived the first $n$ rounds, formally
\begin{equation}\label{N_definition}
N_{M}^{(0)}:=M\quad\text{and}\quad N_{M}^{(n)}:=\xi^{(n)}_{1}+\ldots+\xi^{(n)}_{N_{M}^{(n-1)}}
\end{equation}
for $M,n\in\N$.
\item $1 = S_{1}^{(n)} < S_{2}^{(n)} < S_{3}^{(n)} <\ldots $, the original numbers of the players who survived the first $n$ rounds, formally
\begin{equation}\label{S_definition}
S_{j}^{(n)}:=\inf\{i\in\N\colon N_{i}^{(n)}=j\}
\end{equation}
for $j,n\in\N$.
\item $T(M)$, the number of rounds until only one player (which is of course the first one) among $1,2,\ldots,M$ remains, thus
\begin{equation}\label{T_definition}
T(M)\ :=\ \inf\{n\in\N\colon N_{M}^{(n)}=1\}
\end{equation}
for $M\in\N$.
\item $T_{0}(M)$, the number of \emph{conclusive} rounds, i.e. the number of rounds among $1,\ldots,T(M)$ in which at least one of the players with original labels $1,\ldots,M$ is eliminated:
\begin{equation}\label{T0_definition}
T_{0}(M):=\sum_{j=0}^{T(M)-1} \1\{N_{M}^{(j)}\neq N_{M}^{(j+1)}\}.
\end{equation}
\end{itemize}
The counting process of records is defined by
\begin{equation}\label{k_definition}
K(M)\ :=\ \xi_{1}+\xi_{2}+\ldots+\xi_{M},\quad M\in\N,
\end{equation}
and the associated process of record times as the corresponding first passage time process, viz.
\begin{equation}\label{nu_definition}
\nu(k)\ :=\ \inf\{j\in\N\colon K(j)\ge k\}\ =\ \inf\{j\in\N\colon K(j)=k\}, \quad k\in\N.
\end{equation}
Here, the second equality holds because $K(M)$ has unit jumps only. Hence, $K(M)$ is the number of records in the sample of size $M$, whereas $\nu(k)$ is the index of the $k$-th record in the infinite sample.

\section{Main results}\label{sec:results}

In this section we state our main results for the model just introduced. All proofs are presented in Section \ref{sec:proofs}. Whereas  the classical leader-election procedure based on iid Bernoulli trials was mostly studied by tools from Analytic Combinatorics, as in \cite{Fill+Mahmoud+Szpankowski:96,Prodinger:93}, or by exploiting the connection with maxima in samples from the geometric distribution, as in \cite{BrussGrubel:03},  hereby relying on the particularly nice structure of this model, our approach is entirely different and not even restricted to the leader-election procedure proposed here. For example, we could also treat by similar methods  procedures based on independent Bernoulli trials with success probabilities that are given by any sufficiently nice function of $i$ instead of \eqref{xi_distribution}, for example $\theta/(i+\theta-1)$, $\theta>0$, or $i^{-\alpha}$, $\alpha\in (0,1)$.

\vspace{.1cm}
In what follows, we use $\overset{d}{\longrightarrow}$ to denote convergence in distribution and ${\overset{f.d.d.}{\longrightarrow}}$ to denote weak convergence of finite-dimensional distributions. The notation ${\overset{w}{\longrightarrow}}$ is used for weak convergence of random elements on a topological space to be specified on every occurrence. Let $[x]$ denote the integer part of $x$.

\subsection{The time needed to select the leader}

Before stating our limit theorem for $T(M)$ as $M\to\infty$, we provide some intuition. It is known \cite{Renyi:62} that the number $K(M)$ of records in $M$ iid observations satisfies
\begin{equation}\label{slln_records}
\lim_{M\to\infty} \frac{K(M)}{\log M} = 1 \quad \text{a.s.}
\end{equation}
For very large $M$, the number of persons among $1,\ldots,M$ who survive the first round is therefore approximately $\log M$, the number of persons who survive the second round is approximately $\log\log M$, and so on. In order to obtain an approximation for $T(M)$, consider the \textit{iterated logarithm} $\log^{*}$ which is the integer-valued function defined as
\begin{equation}\label{log-star-definition}
\log^{*}x := 0,\quad 0 \le x < 1,
\quad
\log^{*}x := 1+\log^{*}(\log x),\quad x\ge 1.
\end{equation} 
More explicitly, we have
$$
\log^{*}x = 
\begin{cases}
0, & \text{ if } 0 \le x < 1,\\
1, & \text{ if } 1\le x < e,\\
2, & \text{ if } e\le x < e^e,\\
\vdots &\vdots\\
j, & \text{ if } e\uparrow\uparrow (j-1) \leq x < e\uparrow\uparrow j, \;\; j\in\N,\\
\vdots &\vdots,
\end{cases}
$$
where we used Knuth's uparrow notation \cite[Ch.~2]{Knuth:96}
$$
a\uparrow\uparrow b = a^{\,{\textstyle a}^{\iddots^{\textstyle a}}} \quad \text{($b$ copies of  $a$)}.
$$
The previous considerations suggest that, for large $M$, $T(M)$ should differ from $\log^* M$ by at most $O(1)$. This is confirmed by our first theorem.

\begin{Theorem}\label{mainT0}
The sequence $(T(M) -\log^{*} M)_{M\in \N}$ is tight (and in fact $L^{r}$-bounded for every $r > 0$), but it does not converge in distribution.
\end{Theorem}

In view of this result, it is natural to ask for the subsequential distributional limits of the above sequence for which we will need the \emph{modified iterated exponentials} (or the \emph{modified tetration}), for $\rho\ge 1$ recursively defined by
\begin{equation}\label{eq:def_E_{n}}
{E}_{0}(\rho)\,:=\,\rho\quad\text{and}\quad
{E}_{n}(\rho)\,:=\,e^{{E}_{n-1}(\rho)-1}\quad\text{for }n\in\N.
\end{equation}
The subtraction of $1$ in the above definition ensures that each ${E}_{n}$ is a strictly increasing continuous self-map of the interval $[1,+\infty)$ and $E_{n}(1)=1$. The standard tetration without this subtraction will be discussed in Subsection \ref{sec::st_tetration}.

\begin{Theorem}\label{mainT}
The following convergence of finite-dimensional distributions along the subsequence $([E_{n}(\rho)])_{n\in\N}$ holds true:
\begin{equation}\label{eq:T_fdd_conv}
(T([E_{n}(\rho)])-n)_{\rho>1}\ \todistrfd\ (T^{*}(\rho))_{\rho>1}\quad\text{as }n\to\infty,
\end{equation}
where $(T^{*}(\rho))_{\rho>1}$ is a stochastically continuous process with nondecreasing, c\`adl\`ag sample paths satisfying the stochastic fixed-point equation
\begin{equation}\label{T_equation}
(T^{*}(e^{\rho-1}))_{\rho>1}\ \ofdd\ (T^{*}(\rho) + 1)_{\rho>1}.
\end{equation}
In particular, for each $\rho>1$,
\begin{equation}\label{eq:subseq_limit_T}
T([E_{n}(\rho)])-n\ \todistr\ T^{*}(\rho)\quad\text{as }n\to\infty.
\end{equation}
The random variables $T^{*}(\rho)$, $\rho>1$, are integer-valued with $\Prob\{T^{*}(\rho)=k\}>0$ for all $k\in\Z$ and further pairwise distinct in law.
\end{Theorem}

\begin{Rem}\label{rem0:non-convergence}\rm
The fact that $T(M)-\log^{*}M$ is tight but \emph{not} convergent in distribution should not be surprising and matches a similar result by Fill et al. \cite[Cor.~1]{Fill+Mahmoud+Szpankowski:96} for the classical leader-election algorithm which uses iid Bernoulli trials instead of record times to determine the leader. The phenomenon is due to periodic fluctuations, and Theorem \ref{mainT} provides the right way of scaling (via modified tetrations) so as to identify all subsequential distributional limits. A similar result could be stated for the classical algorithm, but the scaling would instead be in terms of iterations of a linear function, for in this context the number of eliminated persons per round is approximately a fixed fraction $p$, where $p$ is determined by the chosen coin. Periodic fluctuations, typically of a very small order of magnitude, are pervasive in the analysis of algorithms and digital trees, see e.g. \cite{Flajolet+Roux+Vallee:10}, but also in other fields like the theory of branching processes, see e.g. \cite{AlsRoe:96,BigginsBingham:91,BigginsNada:93}.
\end{Rem}

\begin{Rem}\label{rem1:mainT}\rm
Regarding \eqref{eq:T_fdd_conv}, we point out that for a stochastic process with nondecreasing, c\`adl\`ag sample paths, convergence of finite-dimensional distributions implies weak convergence in the Skorokhod $D$-space endowed with the $M_{1}$-topology. On the other hand, provided that the limiting process is continuous in probability, it was claimed in~\cite[Thm.~3]{Bingham:72} but disproved in~\cite{Yamazato:09} that this holds even true if the Skorokhod space is endowed with the $J_{1}$-topology.
\end{Rem}

\begin{figure}
\centering
     \includegraphics[width=10cm]{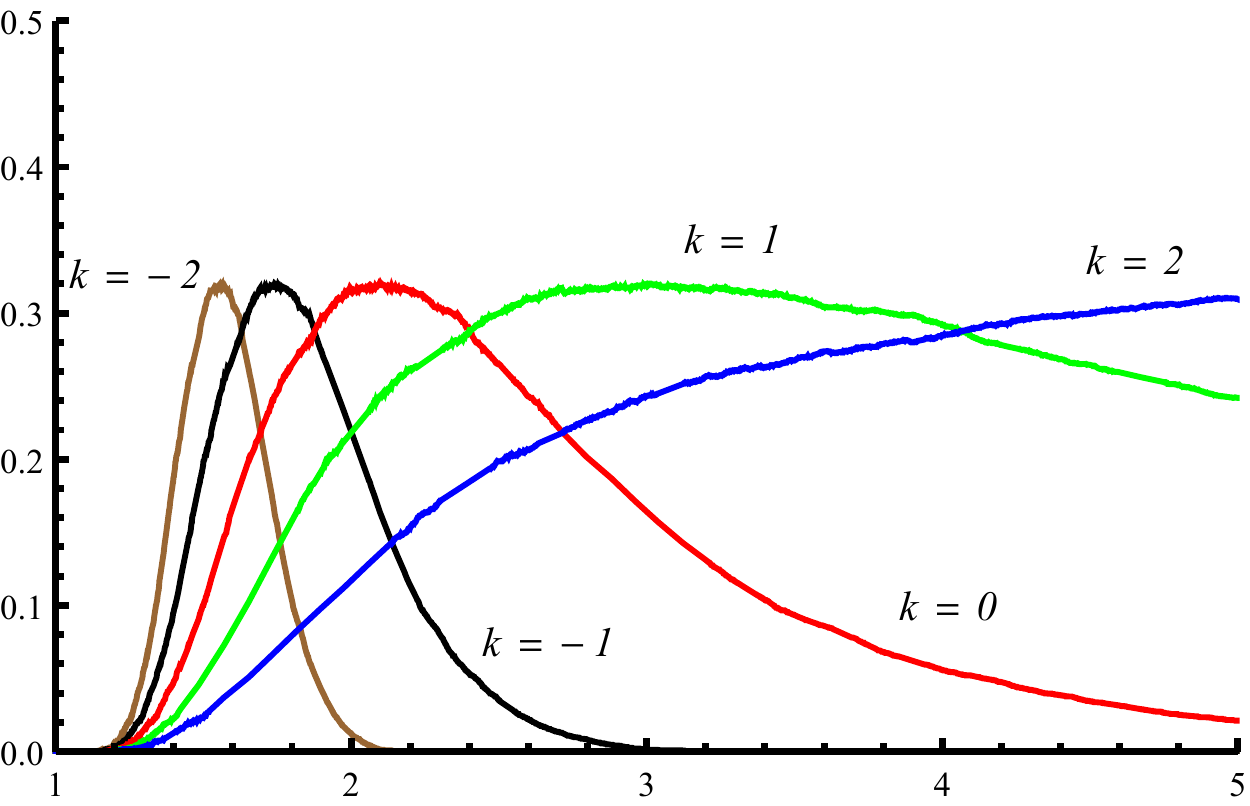}
\caption{The functions $\rho\mapsto \Prob\{T^{*}(\rho)=k\}$ for $k\in\{-2,-1,0,1,2\}$. Any of these curves can be used to generate all others because of the relation~\eqref{T_equation}.}
\label{fig:distr_T_rho}
\end{figure}


\begin{Rem}\label{rem2:mainT}\rm
The limiting process $(T^{*}(\rho))_{\rho>1}$ does not seem to belong to any known family of stochastic processes. As we have not been able to derive a closed-form analytic expression for $\Prob\{T^{*}(\rho) = k\}$, $k\in\Z$, we must leave this as an open problem. Yet, it is possible to compute these probabilities numerically by use of Monte-Carlo simulation; see Figure~\ref{fig:distr_T_rho} and Section~\ref{sec:simulation} for a description of the simulation method. 
\end{Rem}

The next result shows that there are no further weak subsequential limits beyond those given in \eqref{eq:subseq_limit_T}.

\begin{Prop}\label{prop:no_further_limits}
If $T(a_{n})-b_{n}$ converges in distribution to a nondegenerate random variable $Z$, where $(a_{n})_{n\in \N}\subset \N$ and $(b_{n})_{n\in\N}\subset \Z$ are sequences such that $a_{n}\to+\infty$, then there exist $\rho>1$ and  $c\in\Z$ such that
$$ Z\ \eqdist\ T^{*}(\rho) + c. $$
\end{Prop}

Finally, we provide a strong law of large numbers for $T(M)$ that will be deduced from the previous results.

\begin{Theorem}\label{LLN_for_T}
As $M\to\infty$,
$$ \frac{T(M)}{\log^* M}\ \to\ 1\quad\text{a.s. and in }L^{r}\text{ for each }r>0. $$
\end{Theorem}

\subsection{Positions of the persons who survived the first $n$ rounds}

We next turn to the asymptotic behavior of the process $(S_{j}^{(n)})_{j\in\N}$, as $n\to\infty$, recalling that $S_{j}^{(n)}$ is the original label of the person who survived $n$ rounds and whose number after $n$ rounds is $j$. Here we need the \textit{modified iterated logarithms}, for $\rho\ge 1$ recursively defined by
\begin{equation}\label{eq:def_L_{n}}
{L}_{0}(\rho)\,:=\,\rho\quad\text{and}\quad {L}_{n}(\rho)\,:=\,1+\log {L}_{n-1}(\rho)
\quad\text{for }n\in\N.
\end{equation}
Obviously, each $L_{n}$ is a strictly increasing continuous function which maps the interval $[1,+\infty)$ onto itself and $L_n(1)=1$. The functions $L_{n}$ and $E_{n}$ are inverse to each other, viz.
$$ {L}_{n}({E}_{n}(s))\ =\ {E}_{n}({L}_{n}(s))\ =\ s\quad\text{for }s\ge 1. $$
Moreover, for arbitrary integers $n\ge m\ge 0$, we further have
\begin{equation}\label{eq:group_property}
{L}_{n}({E}_m(s))\ =\ {L}_{n-m}(s)\quad\text{and}\quad {L}_m({E}_{n}(s))\ =\ {E}_{n-m}(s),
\quad s\ge 1.
\end{equation}
We therefore define $L_{-n}:= E_{n}$ and $ E_{-n}:=  L_{n}$ for $n\in\N_{0}$ and note that
the set of functions $\{ E_{n}\}_{n\in\Z}$ (or, equivalently, $\{ L_{n}\}_{n\in\Z}$)
forms an infinite cyclic subgroup of the group $(\mathcal{C}^{\uparrow}([1,\infty)),\circ)$ of continuous, strictly increasing functions mapping $[1,\infty)$ onto itself with composition as the group operation.

\begin{Theorem}\label{main1}
There exist random variables $1=S^{*}_{1}\le S^{*}_{2} \le \ldots$ such that
\begin{equation}\label{basic_convergence}
({L}_{n}(S^{(n)}_{1}),{L}_{n}(S^{(n)}_{2}),{L}_{n}(S^{(n)}_{3}),\ldots)\ \todistrfd\  ({S}^{*}_{1},S^{*}_{2},{S}^{*}_{3},\ldots)
\end{equation}
as $n\to\infty$, and the limit vector satisfies the stochastic fixed-point equation
\begin{equation}\label{fp_limit_S}
({S}^{*}_{1},{S}^{*}_{2},{S}^{*}_{3},\ldots)\ \eqdist\ ({L}_{1}({S}^{*}_{\nu(1)}),{L}_{1}({S}^{*}_{\nu(2)}),{L}_{1}({S}^{*}_{\nu(3)}),\ldots),
\end{equation}
where the sequence $(\nu(j))_{j\in\N}$ of record times on the right-hand side is independent of $({S}^{*}_{j})_{j\in\N}$ and given by \eqref{nu_definition}.
\end{Theorem}

\begin{Rem}\label{rem:main1}\rm
Note that the normalizing functions in Theorem \ref{main1} are ${L}_{n}$ rather than
\begin{equation*}
\widetilde L_{n}(\cdot) := \underbrace{\log\circ\ldots\circ\log}_{n \text{ times}}(\cdot),
\end{equation*}
which the reader might deem more natural. However, the choice of $\widetilde L_{n}$ would lead to inevitable complications because $\widetilde L_{n}(S_{j}^{(n)})$ is undefined on a set of positive probability, viz.
$$ \left\{S_{j}^{(n)}\leq \underbrace{\exp\circ\ldots\circ\exp}_{n-2 \text{ times}}(1)\right\}. $$
We refer to Subsection~\ref{sec::st_tetration} for further discussion.
\end{Rem}

\begin{Rem}\rm
Standard arguments (see \cite[Appendix D, Thm.~D.1 and Cor.~D.2]{Sigman:95}) imply that Theorem~\ref{main1} can be stated in terms of point processes as follows. On the space of point measures on $[1,+\infty)$ endowed with the vague topology, the following weak convergence holds true:
$$ \sum_{j=1}^{\infty} \delta_{{L}_{n}({S}_{j}^{(n)})}\ \toweak\ \sum_{j=1}^{\infty} \delta_{ {S}^{*}_{j}}\quad\text{as }n\to\infty, $$
where $\delta_x$ denotes the Dirac point mass at $x$.
\end{Rem}

\begin{figure} [ht]
\centering
     \includegraphics[width=10cm]{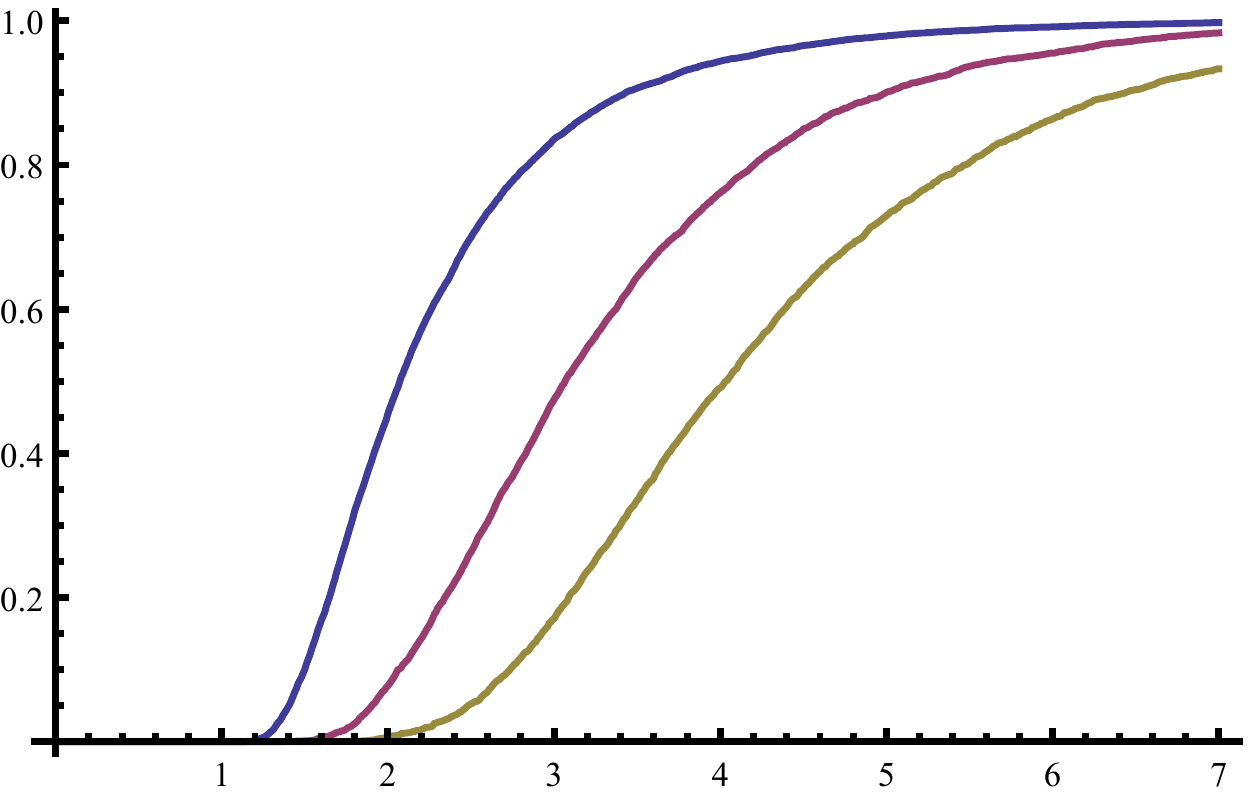}
\caption{Distribution functions of  $S_{2}^{*}$, $S_{3}^{*}$, $S_{4}^{*}$ (from the left to the right).}
\label{fig:distr_funct_S}
\end{figure}


Since the constant vector $(1,1,\ldots)$ also satisfies the fixed-point equation \eqref{fp_limit_S}, one may wonder whether the limit vector in
Theorem \ref{main1} is random. The next result, giving a strong law of large numbers as well as a central limit theorem for $S_{k}^{*}$, as $k\to\infty$, confirms that the process $({S}_{j}^{*})_{j\in\N}$ is nondegenerate.

\begin{Theorem}\label{SLLN/CLT_S_{n}}
The sequence $(S^{*}_{k})_{k\in\N}$ satisfies
$$ \lim_{k\to\infty}\frac{S^{*}_{k}}{k}\ =\ 1\quad\text{a.s.},\quad\lim_{k\to\infty} \frac{\Erw S_{k}^{*}}{k} =1,\quad\text{and}\quad\frac{S^{*}_{k} - k}{\sqrt k}\ \todistr\ {\rm{N}}(0,1), $$
where ${\rm{N}}(0,1)$ denotes the standard normal distribution.
\end{Theorem}

The next theorem states that, far away from the origin, the spacings of the point process $\sum_{i=1}^{\infty} \delta_{S_{i}^{*}}$ look approximately like the spacings of a standard Poisson point process.

\begin{Theorem}\label{limit_poisson}
Let $\cE_{1},\cE_{2},\ldots$ be independent standard exponential random variables. Then, as $k\to\infty$,
\begin{equation}\label{eq:conv S_j^*}
(S_{k+1}^{*} - S_{k}^{*},S_{k+2}^{*} - S_{k+1}^{*},\ldots,S_{k+m}^{*} - S_{k+m-1}^{*},\ldots)\ \overset{f.d.d.}{\longrightarrow}\ (\cE_{1},\cE_{2},\ldots, \cE_m,\ldots).
\end{equation}
\end{Theorem}

We will frequently make use of the fact, stated in the subsequent proposition, that the law of $S_{k}^{*}$ for $k\ge 2$ (recall that $S_{1}^{*}=1$) is continuous and that, for $m\ge 2$, the joint law of $(S_{2}^{*},\ldots,S_{m}^{*})$ puts positive mass on every open $(m-1)$-dimensional interval centered at some $(x_{2},\ldots,x_{m})$ with $1<x_{2}<\ldots<x_m$.

\begin{Prop}\label{prop:no_atoms}
The law of $S_{k}^{*}$ is continuous for every $k\ge 2$ and
\begin{align}
\Prob\{{S}_{2}^{*} \in [\alpha_{2},\beta_{2}], \ldots, {S}_m^{*} \in [\alpha_m,\beta_m]\}\,>\,0
\label{eq:all values attained}
\end{align}
for all $1<\alpha_{2}<\beta_{2}<\ldots< \alpha_m<\beta_m$ and $m\in\N$.
\end{Prop}

Although the asymptotic properties of the point process $\sum_{i=1}^{\infty}\delta_{S_{i}^{*}}$ resemble those of the standard Poisson process, it seems that $\sum_{i=1}^{\infty}\delta_{S_{i}^{*}}$ does not belong to any known family of point processes.
It is natural to conjecture that the distribution of the random vector $({S}_{2}^{*}, \ldots, {S}_m^{*})$ is absolutely continuous.
We also conjecture that $1 = S_{1}^{*}< S_{2}^{*}<\ldots$ a.s.\ (with strict inequalities).

\begin{Rem}\label{rem:limit_poisson}\rm
There is a connection between the distributions of $S_{2}^{*}$ and $T^{*}(\rho)$, namely
\begin{equation}\label{eq:S_{2}_ast_T_rho_connection}
\Prob\{T^{*}(\rho) \le  k\}\ =\
\Prob\{S_{2}^{*}>L_{k}(\rho)\}, \quad k\in\Z,
\end{equation}
as shown in Step 1 in the Proof of Theorem~\ref{mainT} (Subsection~\ref{subsec:proof_mainT}). We also note in passing that the function $p_0^*(\rho):=\Prob\{T^{*}(\rho) = 0\}$ (displayed in red in Figure~\ref{fig:distr_T_rho}) satisfies the curious identity
$$ \sum_{n\in\Z} p_0^*(E_{n}(\rho)) = 1, \quad \rho>1, $$
for, by~\eqref{T_equation}, $p_0^*(E_{n}(\rho)) = \Prob\{T^{*}(E_{n}(\rho)) = 0\}= \Prob\{T^{*}(\rho)+n=0\}$.
\end{Rem}

A ``space-time version'' of Theorem~\ref{main1} is next. We have already observed that the distribution of $V:=({S}_{j}^{*})_{j\in\N}$ is invariant under the random mapping $\psi: \R^{\N} \to \R^{\N}$
\begin{equation}\label{psi_mapping_def}
(x_{1},x_{2},x_{3},\ldots)\ \mapsto\ (L_{1}(x_{\nu(1)}),L_{1}(x_{\nu(2)}), L_{1}(x_{\nu(3)}),\ldots).
\end{equation}
By iteration of this mapping, we can construct a one-sided stationary random sequence $(V_{k})_{k\in \N_0}$ such that each $V_{k} = ({S}_{k,j}^{*})_{j\in\N}$, $k\in\N_{0}$, is a $\R^{\N}$-valued random vector with the same distribution as $V$, and each $V_{k+1}$ is obtained from $V_{k}$ by applying an independent copy of $\psi$. Moreover, by the Kolmogorov consistency theorem, we can construct a two-sided stationary version $(V_{k})_{k\in \Z}$ such that $V_{k+1} = \psi^{(k)}(V_{k})$ for all $k\in\Z$, where $(\psi^{(k)})_{k\in\Z}$ is a sequence of independent copies of $\psi$.
 \begin{Theorem}\label{main1_time}
Stipulating that $L_{n+k}(S^{(n+k)}_{j}):=0$ for $n+k<0$, we have that
$$ ({L}_{n+k}(S^{(n+k)}_{j}))_{(j,k)\in\N\times\Z}\ \todistrfd\ ({S}^{*}_{k,j})_{(j,k)\in\N\times\Z} $$
as $n\to\infty$.
\end{Theorem}

\subsection{The number of survivors after $n$ rounds}
Recall that $N^{(n)}_{M}$ is the number of persons among $\{1,\ldots,M\}$ who survived the first $n$ rounds. Let ${N}^{*}(\rho)$ denote the counting process associated with the sequence $(S^{*}_{k})_{k\in\N}$ appearing in Theorem \ref{main1}, that is
\begin{equation}\label{N_ast_def}
{N}^{*}(\rho) := \#\{k\in\N\colon {S}_{k}^{*}\le \rho\},\quad \rho\ge 1.
\end{equation}
The next theorem is a ``dual'' version of Theorem~\ref{main1}. 

\begin{Theorem}\label{main2}
The process $({N}^{*}(\rho))_{\rho\ge 1}$ is continuous in probability, and
\begin{equation}\label{eq:N_{n}_conv_fdd}
\left(N^{(n)}_{[{E}_{n}(\rho)]}\right)_{\rho\ge 1}\ \todistrfd\ ({N}^{*}(\rho))_{\rho\ge 1}\quad\text{as }n\to\infty.
\end{equation}
If $(K(M))_{M\in\N}$ is a record counting sequence $($see~\eqref{k_definition}$)$ and independent of $({N}^{*}(\rho))_{\rho\ge 1}$, then the following stochastic fixed-point equation holds:
\begin{equation}\label{eq:fixed_point_N_star}
(K({N}^{*}(\rho)))_{\rho\ge 1}\ \eqdist\ ({N}^{*}(1+\log \rho))_{\rho\ge 1}.
\end{equation}
\end{Theorem}

\begin{Rem}\rm
In contrast to Theorem \ref{mainT} the parameter $\rho$ in Theorem \ref{main2} varies in $[1,\infty)$ rather than $(1,\infty)$. If $\rho=1$ then both sides of \eqref{eq:N_{n}_conv_fdd} equal $1$.
\end{Rem}

\begin{Rem}\label{rem:main2}\rm
The processes in \eqref{eq:N_{n}_conv_fdd} have all nondecreasing, c\`adl\`ag sample paths, and the limiting process $({N}^{*}(\rho))_{\rho\ge 1}$ is continuous in probability. 
Hence, by the same argument as in Remark~\ref{rem1:mainT}, \eqref{eq:N_{n}_conv_fdd} even holds in the sense of weak convergence on $D[1,\infty)$ endowed with the $M_{1}$-topology.
\end{Rem}

Recall from Theorem \ref{main1_time} above that $({S}_{k,j}^{*})_{(j,k)\in\N \times \Z}$ equals the ``space-time'' limit of the random field $(L_{n+k}(S_{j}^{(n+k)}))_{(j,k)\in\N\times\Z}$ as $n\to\infty$. For $\rho\ge 1$ define the counting processes
$$
{N}_{k}^{*}(\rho):=\#\{j\in\N:{S}_{k,j}^{*}\le L_{k}(\rho)\},\quad k\in\Z.
$$
The next result is the ``space-time'' version of Theorem \ref{main2}. 

\begin{Theorem}\label{main2_time}
As $n\to\infty$,
$$
\left(N^{(n+k)}_{[{E}_{n}(\rho)]}\right)_{\rho\ge 1,k\in\Z}\ \todistrfd\ ({N}_{k}^{*}(\rho))_{\rho\ge 1,k\in\Z}.
$$
\end{Theorem}

\subsection{The number of conclusive rounds}

We finally turn to the random variable that will appear in connection with the Poisson-Dirichlet coalescent, namely the number of conclusive rounds $T_{0}(M)$ it takes to eliminate all but the first person among $1,2,\ldots,M$.
Recall from \eqref{T0_definition} that
$$ T_{0}(M)\ :=\ \sum_{j=0}^{T(M)-1} \1\{N_{M}^{(j)}\neq N_{M}^{(j+1)}\}. $$
If the number of players, say $k$, in a round is large, the event that the round is inconclusive has very small probability, namely $1/k!$, however at times close to $T(M)$ inconclusive rounds may appear with non-negligible probability.
Theorem \ref{mainT_{0}} below constitutes a counterpart of Theorem \ref{mainT} for $T_{0}(n)$ and forms a key ingredient to the proof of Theorem \ref{PD-theorem} on the Poisson-Dirichlet coalescent stated below. Before stating it, we give an auxiliary proposition describing the behavior of the model shortly before the time $T([E_{n}(\rho)])$ when all but the first player have left. Fix $m\in\N$ and let
$$ \mathcal{Z}_{m,n}(\rho)\ :=\ \Big(N^{(T([E_{n}(\rho)])-1)}_{[E_{n}(\rho)]},N^{(T([E_{n}(\rho)])-2)}_{[E_{n}(\rho)]},\ldots,N^{(T([E_{n}(\rho)])-m)}_{[E_{n}(\rho)]}\Big).
$$
Note that the first coordinate is the number of players among $1,2,\ldots,[E_{n}(\rho)]$ who participate in the last round, the second coordinate is the
number of players among $1,2,\ldots,[E_{n}(\rho)]$, who participate in the penultimate round, and so on.

\begin{Prop}\label{Proposition_last_rounds}
For every $m\in\N$,
$$ \Big(T([E_{n}(\rho)])-n,\mathcal{Z}_{m,n}(\rho)\Big)_{\rho>1}\ \todistrfd\ \Big(T^{*}(\rho),\mathcal{Z}^{*}_{m}(\rho)\Big)_{\rho>1}, $$
where $T^{*}(\rho)$ is as in Theorem~\ref{mainT} and
$$ \mathcal{Z}^{*}_{m}(\rho)\ :=\ \Big(N^{*}_{T^{*}(\rho)-1}(\rho),\ldots, N_{T^{*}(\rho)-m}^{*} (\rho)\Big). $$
\end{Prop}

The next theorem is the basic convergence result on the number of conclusive rounds.
\begin{Theorem}\label{mainT_{0}}
As $n\to\infty$, the following convergence of the finite-dimensional distributions holds:
$$ (T_{0}([E_{n}(\rho)])-n)_{\rho>1}\ \todistrfd\ (T_{0}^{*}(\rho))_{\rho>1}, $$
where
$$
T_{0}^{*}(\rho)\
:=\
T^{*}(\rho)-\sum_{j=-\infty}^{T^{*}(\rho)-1}\1\left\{N^{*}_{j}(\rho)=N^{*}_{j+1}(\rho)\right\}.
$$
In particular, for each $\rho>1$,
$$
T_{0}([E_{n}(\rho)])-n\ \todistr\ T_{0}^{*}(\rho)\quad\text{as }n\to\infty.
$$
\end{Theorem}

Some properties of the stochastic process $(T_0^{*}(\rho))_{\rho>1}$ are collected in the next proposition.

\begin{Prop}\label{prop:T_0_ast_properties}
The random variables $T_{0}^{*}(\rho)$, $\rho>1$, are integer-valued and nondegenerate. For $1<\rho_{1}<\rho_{2}$ we have $T_0^{*}(\rho_{1}) \leq T_0^{*}(\rho_{2})$ and the inequality is strict on an event of positive probability. In particular, the distributions of $T^{*}_0(\rho)$, $\rho>1$, are pairwise distinct.
\end{Prop}


\subsection{Passing to the standard tetration}\label{sec::st_tetration}

With some adjustments to be explained next, it is possible to replace $E_{n}(\rho)$ in our results by the simpler iterations of exponentials $\widetilde E_n(\rho)$, called \emph{standard tetration}, viz.
\begin{equation}\label{eq:standard tetration}
\widetilde E_{0}(\rho)\,:=\,\rho\quad\text{and}\quad
\widetilde E_{n}(\rho)\,:=\,\underbrace{\exp\circ\ldots\circ\exp}_{\textrm{$n$ times}}(\rho)
\quad\text{for }n\in\N\,\text{ and }\,\rho\in\R.
\end{equation}
Define the function $f$ by
\begin{equation}\label{f_def}
f(z)\,:=\,\lim_{n\to\infty}{L}_{n}(\widetilde E_{n}(z)), \quad z\in\R.
\end{equation}
Lemma \ref{f_existence} in the Appendix shows that $f$ is well-defined for all $z\in\R$, continuous and strictly increasing on $\R$. It follows from the definition that $f$ satisfies the functional equation
\begin{equation}\label{f_equation}
f(z)=1+\log f(e^z), \quad z\in\R,
\end{equation}
and it conjugates the dynamics generated by the mappings $z\mapsto E_{1}(z)$ and $z\mapsto\widetilde{E}_{1}(z)$ in the sense that
\begin{equation}\label{f_conjugates}
{E}_{n}(f(z))=f(\widetilde E_{n}(z)),\quad\text{for all }n\in\Z\text{ and }z\in\R.
\end{equation}
We refer to \cite[Chapter 8]{KuczmaChoGer:90} for a general theory of such conjugacy relations.

The monotonicity of $f$ ensures that the limit $f(-\infty):=\lim_{z\to-\infty} f(z)$ exists and is $\ge 1$. Numerical calculations show that actually
$$
f(-\infty)\ =\ 1+\log f(0)\ \approx\ 1.6130198923451345686407,
$$
and so the image of $\R$ under $f$ is slightly smaller than the interval $(1,\infty)$.

\begin{figure} [t]
\centering
     \includegraphics[width=10cm]{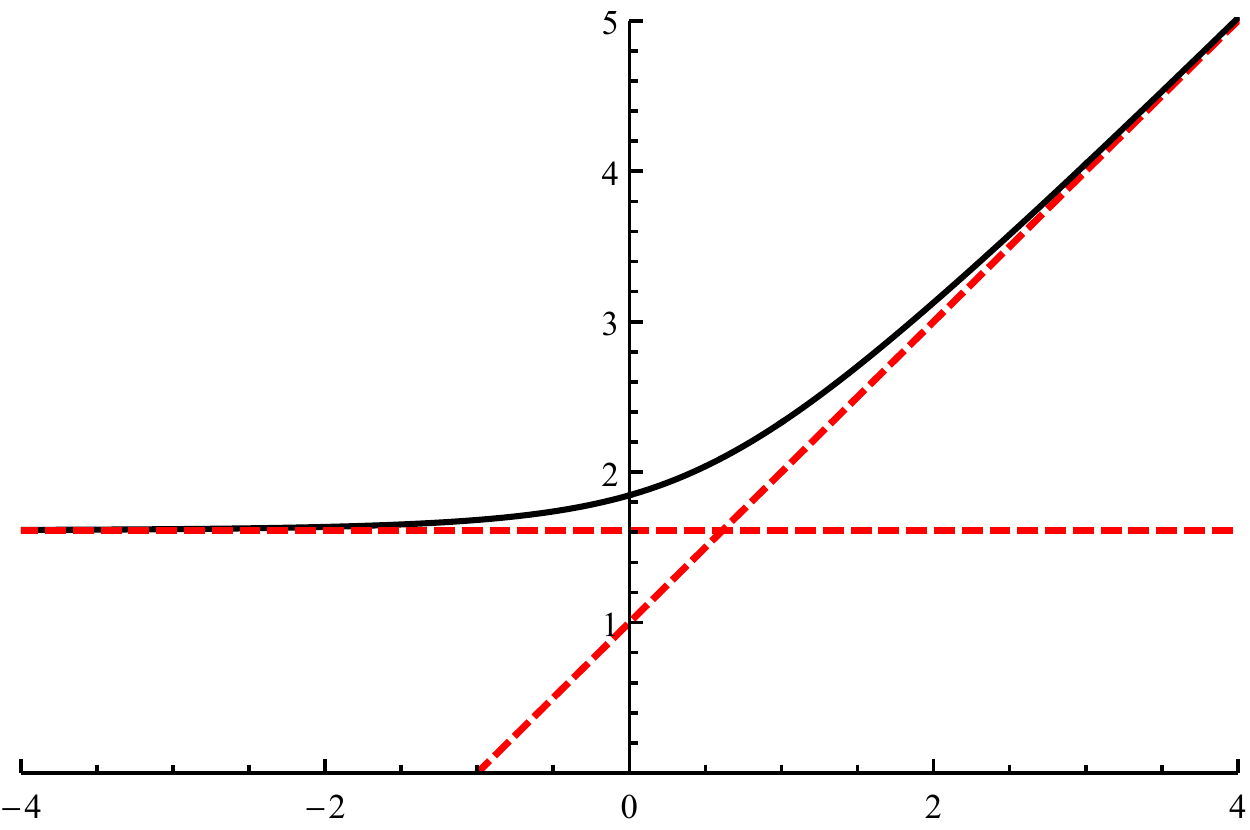}
\caption{The graph of the function $f$ (solid black) together with the asymptotes $y=f(-\infty)$ and $y=x+1$ (dashed red)}.
\label{fig:f}
\end{figure}

Note that in the subsequent corollaries on convergence along $\widetilde E_{n}(\rho)$, the argument $\rho$ takes values in $\R$ rather than in $(1,\infty)$ as in the previous theorems.

\begin{Cor}\label{mainT_tilde}
With $\widetilde T^{*}(\rho):=T^{*}(f(\rho))$ for $\rho\in\R$, we have
\begin{equation*}
\left(T([\widetilde E_{n}(\rho)])-n\right)_{\rho\in\R}\ \todistrfd\ (\widetilde T^{*}(\rho))_{\rho\in\R}\quad\text{as }n\to\infty.
\end{equation*}
\end{Cor}
\begin{Cor}\label{main2p}
Put $\widetilde N^{*}(\rho):={N}^{*}(f(\rho))$ and $\widetilde N^{*}_{k}(\rho):= {N}^{*}_{k}(f(\rho))$ for $\rho\in\R$. Then, as $n\to\infty$,
\begin{align}
\left(N^{(n)}_{[\widetilde E_{n}(\rho)]}\right)_{\rho\in\R}\ &\todistrfd\ (\widetilde N^{*}(\rho))_{\rho\in\R}\label{N_n_fdd_tilde}
\shortintertext{as well as}
\left(N^{(n+k)}_{[\widetilde E_{n}(\rho)]}\right)_{\rho\in\R,k\in\Z}\ &\todistrfd\ (\widetilde N_{k}^{*}(\rho))_{\rho\in\R,k\in\Z}.\nonumber
\end{align}
Moreover, the limit process $(\widetilde N^{*}(\rho))_{\rho\in\R}$ satisfies the stochastic fixed-point equation (compare \eqref{eq:fixed_point_N_star})
\begin{equation*}
\big(K(\widetilde N^{*}(\rho))\big)_{\rho>0}\ \eqdist\ \big(\widetilde N^{*}(\log \rho)\big)_{\rho>0}.
\end{equation*}
\end{Cor}

\begin{Cor}\label{Proposition_last_rounds_p}
For $m\in\N$, $\rho\in\R$ let
$$ \widetilde{\mathcal{Z}}_{m,n}(\rho)\ :=\ \Big(N^{(T([\widetilde E_{n}(\rho)])-1)}_{[\widetilde E_{n}(\rho)]},N^{(T([\widetilde E_{n}(\rho)])-2)}_{[\widetilde E_{n}(\rho)]},\ldots,N^{(T([\widetilde E_{n}(\rho)])-m)}_{[\widetilde E_{n}(\rho)]}\Big). $$
Then
$$ \Big(T([\widetilde E_{n}(\rho)])-n,\widetilde{\mathcal{Z}}_{m,n}(\rho)\Big)_{\rho\in\R}\ \todistrfd\ \Big(\widetilde T^{*}(\rho),\widetilde{\mathcal{Z}}^{*}_{m}(\rho)\Big)_{\rho\in\R}, $$
where $\widetilde T^{*}(\rho)$ is as in Corollary \ref{mainT_tilde} and $\widetilde{\mathcal{Z}}^{*}_{m}(\rho):=\mathcal{Z}^{*}_{m}(f(\rho))$.
\end{Cor}

\begin{Cor}\label{mainT_{0}_tilde}
As $n\to\infty$ and with $\widetilde T_{0}^{*}(\rho):=T_{0}^{*}(f(\rho))$, the following convergence of the finite-dimensional distributions holds:
$$ \left(T_{0}([\widetilde E_{n}(\rho)])-n\right)_{\rho\in\R}\ \todistrfd\ (\widetilde T_{0}^{*}(\rho))_{\rho\in\R}. $$
The random variables $\widetilde T_{0}^{*}(\rho)$, $\rho\in\R$, are integer-valued, nondegenerate and have pairwise different distributions.
\end{Cor}

All previous results are checked in an essentially analogous way by drawing on the results in the preceding subsections. This will be exemplified in Subsection \ref{subsec:main2p} by giving a detailed argument for \eqref{N_n_fdd_tilde} of Corollary \ref{main2p}.

\section{The Poisson-Dirichlet coalescent}\label{sec:PD}

The leader-election procedure based on records arises quite naturally when studying the number of collisions in the \emph{Poisson-Dirichlet coalescent} as we will briefly explain next. An exchangeable coalescent with multiple collisions is a continuous-time Markov process $\Pi_{n}:=(\Pi_{n}(t))_{t\ge 0}$ taking values in the set of partitions of $[n]:=\{1,2,\ldots,n\}$. Starting from the trivial partition into the singletons $\Pi_{n}(0)=\{\{1\},\{2\},\ldots,\{n\}\}$, it evolves according to the following rule. If at some time $t\ge 0$ the number of blocks in the current partition equals $m$, then every $k$-tuple of blocks merges into one block with intensity
$$
\lambda_{m,k}=\int_{0}^{1} x^{k-2}(1-x)^{n-k}\Lambda({\rm d}x),\quad 2\le k\le m,
$$
where $\Lambda$ is a finite measure on $[0,1]$. In view of this characterization exchangeable coalescents with multiple collisions are also called $\Lambda$-coalescents. The special choice of intensities is necessary and sufficient for the consistency of the family of exchangeable processes $(\Pi_{n})_{n\in\N}$. The latter property allows the construction of the infinite coalescent $\Pi:=(\Pi(t))_{t\ge 0}$, the process
with state space of partitions of $\N$ such that the restriction of $\Pi$ to $[n]$ is $\Pi_{n}$ for every $n$. We refer to the seminal paper \cite{Pitman:99} and to the surveys \cite{NBerestycki:09,GneIksMar:14} for further information on $\Lambda$-coalescents.

The definition of the $\Lambda$-coalescent allows only one collision at a time. In \cite{MoehleSagitov:01}, see also \cite{Schweinsberg:00}, the concept of an exchangeable coalescent with {\it simultaneous} multiple collisions was proposed for which the evolution proceeds as follows: if at some time $t\ge 0$ the number of blocks equals $m$, they
merge into $j$ blocks consisting of $k_{1},k_{2},\ldots,k_{j}$ initial blocks ($k_{1}+k_{2}+\ldots+k_{j}=m$ and $k_{1}\ge k_{2}\ge \cdots \ge k_{j}$, $k_{1}\ge 2$) with intensity
\begin{align}\label{merge_{i}ntensities}
&\psi_{j}(k_{1},k_{2},\ldots,k_{j})\ =\ \int_{\Delta^{*}}\sum_{\overset{{i_{1},\ldots,i_{j}\in\N}}{\textrm{\tiny all distinct}}}x_{i_{1}}^{k_{1}}\cdots x_{i_{j}}^{k_{j}}\ \frac{\Xi(dx)}{(x,x)},
\shortintertext{where}
&\Delta^{*}\ :=\ \left\{x=(x_{1},x_{2},\ldots):x_{1}\ge x_{2}\ge \cdots\ge 0,\sum_{i=1}^{\infty}x_{i}=1\right\},\nonumber
\end{align}
$(x,x):=\sum_{i=1}^{\infty}x_{i}^2$ and $\Xi(\cdot)$ denotes a finite measure on $\Delta^{*}$. Exchangeable coalescents with simultaneous multiple collisions are called $\Xi$-coalescents.\footnote{Formula \eqref{merge_{i}ntensities} for the intensities is not the most general one ensuring consistency because the measure $\Xi$ may be supported by a larger simplex $\Delta:=\{x=(x_{1},x_{2},\ldots):x_{1}\ge x_{2}\ge \cdots\ge 0,\sum_{i=1}^{\infty}x_{i}\le 1\}$ and the formula for $\psi_{j}$ may be more involved (see Eq.\ (11) in \cite{Schweinsberg:00}). If $\Xi$ is concentrated on $\Delta^{*}$, this formula reduces to \eqref{merge_{i}ntensities}, see (2) in \cite{Moehle:10}.} The most widely known examples of non-trivial probability measures on $\Delta^{*}$ are the Poisson-Dirichlet distributions ${\rm PD}_{\theta}$ for $\theta>0$, and the Poisson-Dirichlet coalescent $(\Pi^{\theta}_{n}(t))_{t\ge 0}$ is the $\Xi$-coalescent when choosing
$$ \Xi({\rm d}x)/(x,x)\ =\ {\rm PD}_{\theta}({\rm d}x). $$
This parametric family was introduced in \cite{Sagitov:03} where these coalescents appear as the limiting processes for an exchangeable reproduction model described by a compound multinomial distribution, see \cite[Section 3]{Sagitov:03} for details.

\vspace{.1cm}
We are aware of two works \cite{Marynych:10b,Moehle:10} on the asymptotic analysis (for large $n$) of relevant functionals of the Poisson-Dirichlet coalescent like the total number of mergers (with simultaneous mergers counted as one). If $X_{\theta}(n)$ denotes this number when starting with $n$ blocks, it was shown in \cite{Marynych:10b} that for all $\theta>0$ and as $n\to\infty$
\begin{equation}\label{wlln_pd_coal}
\frac{X_{\theta}(n)}{\log^{*}_{\theta}(n)} \to 1 \quad \text{in } L^{r} \text{ for all } r>0,
\end{equation}
where the integer-valued function $\log^{*}_{\theta}(\cdot)$ is defined by
$$ \log^{*}_{\theta}(x) := 0,\quad x\in [0,x_{0}),\quad \log^{*}_{\theta}(x) := 1+\log^{*}_{\theta}(\theta\log x),\quad x\ge x_{0}, $$
for an arbitrary $x_{0}>e^{2\theta\vee 1}$ (this ensures that iterations will eventually end up in $[0,x_{0})$). For $\theta=1$, $\log^{*}_{1}$ can be replaced by the iterated-logarithm defined in \eqref{log-star-definition}, and therefore \eqref{wlln_pd_coal} restated as
\begin{equation}\label{wlln_pd_coal2}
\frac{X_{1}(n)}{\log^{*}n} \to 1 \quad \text{in } L^{r} \text{ for all } r>0.
\end{equation}
Obviously, this result leaves open the question about the asymptotic behavior of the distribution of $X_{1}(n)$ which has been the initial motivation for the present work. By suitable translation, our results on the leader-election procedure will enable us to establish the following theorem which states convergence in law of $X_{1}(k_{n})$ along suitable subsequences $k_{n}$. These subsequences are (the integer parts of) the standard tetration $\widetilde E_{n}(\rho)$  defined in \eqref{eq:standard tetration} (the modified tetration $E_n(\rho)$ could be used as well).

\begin{Theorem}\label{PD-theorem}
For any fixed $\rho\in\R$,
$$ X_{1}([\widetilde E_{n}(\rho)])-n\ \todistr\ \widetilde{T}_{0}^{*}(\rho)\quad\text{as }n\to\infty, $$
where the $\widetilde{T}_{0}^{*}(\rho)=T_{0}^{*}(f(\rho))$, $\rho\in\R$, are the integer-valued, nondegenerate random variables defined in Theorem~\ref{mainT_{0}} and Corollary \ref{mainT_{0}_tilde}.
\end{Theorem}

The proof is based on Theorem \ref{mainT_{0}} on the number $T_{0}(n)$ of conclusive rounds in the leader-election model, the connection following from the fact, to be established in Subsection \ref{subsec:PD-theorem-T_0},  that $X_{1}(n)$ and $T_{0}(n)$ have the same distribution. Since the $\widetilde{T}_{0}^{*}(\rho)$, $\rho\in\R$, have pairwise distinct laws (Prop.~\ref{prop:T_0_ast_properties}), it is clear that the family $X_{1}(M) - \log^{*}M$, $M\in\N$, cannot be convergent in distribution. On the other hand, our last result shows that this family is tight.

\begin{Theorem}\label{PD-theorem_tight}
The sequence $(X_{1}(M) - \log^{*} M)_{M\in\N}$ is tight (and in fact even bounded in $L^{r}$ for every $r>0$), but it does not converge in distribution.
\end{Theorem}

Although the previous results are stated only for the Poisson-Dirichlet coalescent with parameter $\theta=1$, we point out that they can easily be extended to the case of general $\theta>0$. Let us briefly sketch the corresponding construction. Fix $\theta>0$ and consider a $\theta$-modified leader-election procedure where, instead of \eqref{xi_distribution}, the positions of remaining players are determined by the independent random variables
$$
\xi_{i}^{(n)}\ :=\ \1\{\text{in round }n\text{ player with the current number }i\text{ survives}\}
$$
with distribution
$$
\Prob\{\xi_{i}^{(n)}=1\}\ =\ 1-\Prob\{\xi_{i}^{(n)}=0\}\ =\ \frac{\theta}{i+\theta-1}.
$$
for $i,\,n\in\N$. Using exactly the same arguments as in the proof of Theorem \ref{PD-theorem}, in particular formula \eqref{PD_dec_distr}, it can be easily checked that the random variable $X_{\theta}(n)$ has the same distribution as the number of conclusive rounds in the $\theta$-modified leader-election procedure starting with $n$ players. All the results of Section \ref{sec:results} can be extended to the $\theta$-modified leader-election procedure with appropriate adjustments. For example one has to replace tetrations
$E_n$ and iterated logarithms $L_n$ by their $\theta$-modified analogs:
\begin{align*}
{E}^{\theta}_{0}(\rho)\,&:=\,\rho\quad\text{and}\quad{E}^{\theta}_{n}(\rho)\,:=\,e^{\theta^{-1}({E}^{\theta}_{n-1}(\rho)-1)}\quad\text{for }n\in\N;\\
{L}^{\theta}_{0}(\rho)\,&:=\,\rho\quad\text{and}\quad{L}^{\theta}_{n}(\rho)\,:=\,1+\theta\log {L}^{\theta}_{n-1}(\rho)\quad\text{for }n\in\N;\\
{\widetilde{E}}^{\theta}_{0}(\rho)\,&:=\,\rho\quad\text{and}\quad{\widetilde{E}}^{\theta}_{n}(\rho)\,:=\,e^{\theta^{-1}{\widetilde{E}}^{\theta}_{n-1}(\rho)}\quad\text{for }n\in\N;	
\end{align*}
and, as suggested by \eqref{wlln_pd_coal}, $\log^{\ast}$ by $\log_{\theta}^{*}$.

\vspace{.1cm}
We close this section with some further remarks. The quite exotic asymptotic behavior of $X_n$ involving iterated logarithms and tetrations is not very common
in the probabilistic literature, yet it has appeared in several places in population genetics and coalescent theory. In particular, the ``log-star'' asymptotics arise in the analysis of the number of distinguishable alleles in the Ohta--Kimura model of neutral mutation \cite{Kesten:80} and in the analysis of the spatial Kingman coalescent \cite{Angel+Berestycki+Limic:12}. Outside Mathematical Biology the ``log-star'' function has appeared in some problems related to the complexity analysis of computer algorithms \cite{Devillers:92,Linial:92,Schneider+Wattenhofer:08}. Let us finally mention a recent preprint \cite{AlsKabMar:16} where we have studied another generalization of the classical leader-election procedure.

\section{The proofs ...}\label{sec:proofs}

Since iterations of functions will play a role in various places of this section, let us introduce the following shorthand notation: Given a sequence of self-maps $\vph^{(n)}$, $n\in\Z\text{ or }\in\N$, of an arbitrary set $\cS$, we put
$$ \vph^{(k\uparrow n)}(\cdot):=\vph^{(k)}\circ\ldots\circ\vph^{(n)}(\cdot)\quad\text{and}\quad \vph^{(n\downarrow k)}(\cdot):=\vph^{(n)}\circ\ldots\circ\vph^{(k)}(\cdot) $$
for $k\leq n$. For $k<n$, we stipulate that $\vph^{(n\uparrow k)}$ and $\vph^{(k\downarrow n)}$ denote the identity map on $\cS$ if $k<n$. 

\subsection{... of Theorem~\ref{main1}}\label{subsec:main1}
We start with the proof of Theorem \ref{main1}, which is the basic convergence result in this model.

\vspace{.2cm}
For $k,m,n\in\N$, define $K^{(n)}(m)$ as the number of records among the first $m$ players still in the game in round $n$ and $\nu^{(n)}(k)$ as the index of the $k$-th record in round $n$, formally
$$ K^{(n)}(m)\,:=\,\sum_{j=1}^{m}\xi_{j}^{(n)}\quad\text{and}\quad\nu^{(n)}(k)\,:=\,\inf\{j\in\N:K^{(n)}(j) = k\}. $$
Observe that
$$ (S^{(n)}_{1},S^{(n)}_{2},S^{(n)}_{3},\ldots)\ =\ (S^{(n-1)}_{\nu^{(n)}(1)},S^{(n-1)}_{\nu^{(n)}(2)},S^{(n-1)}_{\nu^{(n)}(3)},\ldots),\quad n\in\N. $$
Iterating this relation and using the initial condition $S_j^{(0)}=j$, we obtain
\begin{equation*}
S^{(n)}\ :=\ \big(S^{(n)}_{1},S^{(n)}_{2},\ldots\big)\ =\ \big(\nu^{(1\uarr n)}(1),\nu^{(1\uarr n)}(2),\ldots\big).
\end{equation*}
This shows that the $(S^{(n)}_{1},S^{(n)}_{2},S^{(n)}_{3},\ldots)$ forms an \emph{iterated function system}, obtained by the \textit{forward} iteration of iid copies $\phi^{(1)},\phi^{(2)},\ldots$ of the random mapping
$$ \phi:\ (x_{1},x_{2},x_{3},\ldots)\ \mapsto\ x\circ\nu\ =\ (x_{\nu(1)},x_{\nu(2)},x_{\nu(3)},\ldots) $$
with initial condition $(S^{(0)}_{1},S^{(0)}_{2},S^{(0)}_{3},\ldots)=(1,2,3,\ldots)$, thus
$$
S^{(n)}\ =\ \phi^{(n\darr 1)}(S^{(0)}). 
$$
for $n\in\N$. A standard tool for such systems is to pass to the \emph{backward} iterations. This transformation does not change the distributions, hence
\begin{equation}\label{eq:forward_backward_eqdistr}
S^{(n)}\
\eqdist\
\phi^{(1\uarr n)}(S^{(0)})
\ =\
\big(\nu^{(n\darr 1)}(1),\nu^{(n\darr 1)}(2),\ldots\big)
\end{equation}
for every $n\in\N$.
The advantage of this passage lies in the fact that upon applying the mapping ${L}_{n}$ to the coordinates on the right-hand side of~\eqref{eq:forward_backward_eqdistr}
we obtain an \textit{almost surely} convergent sequence, while doing the same with the forward iterations gives convergence in distribution only. 
For an excellent general survey of iterated function systems and applications, we refer to Diaconis and Freedman \cite{DiaconisFr:99}. Define
\begin{equation}\label{eq:def_eta}
\eta_{j}^{(0)}\,:=\,j\quad\text{and}\quad\eta_{j}^{(n)}\,:=\,\nu^{(n\darr 1)}(j)\quad\text{for }j,n\in\N.
\end{equation}
Our aim is to show that, for every fixed $j\in\N$, the sequence
\begin{equation}\label{eq:def_X}
\cX_{n,\,j}\,:=\,{L}_{n}(\eta_{j}^{(n)}), \quad n\in\N_{0},
\end{equation}
converges a.s.,\ as $n\to\infty$.  To this end, it suffices to show that the sequence $\cX_{n,\,j}$ has almost surely bounded variation for arbitrary fixed $j\in\N$, i.e.\
\begin{equation}\label{X_bv}
\sum_{n=1}^{\infty}|\cX_{n,\,j}-\cX_{n-1,\,j}|<\infty \quad \text{a.s.}
\end{equation}
This being obvious for $j=1$ because $\cX_{n,\,1}=1$ for all $n\in\N_{0}$, let $j\in \{2,3,\ldots\}$ and note that, by Prop.~\ref{nu_{n}_prop} in the Appendix,
$$ \sup_{m\in\N} \Erw\left(\frac{m}{{L}_{1}(\nu(m))}\right)^{r}\,<\,C_{r}\quad\text{and}\quad
\sup_{m\in\N} \Erw \left|\frac{{L}_{1}(\nu(m))-m}{\sqrt{m}}\right|^{r}\,<\,C_{r} $$
for arbitrary $r>0$ and some finite constant $C_{r}$ depending only on $r$.
Since $\eta_{j}^{(n-1)}$ is independent of $\nu^{(n)}$ and since $\eta_j^{(n)} = \nu^{(n)}(\eta_j^{(n-1)})$, see~\eqref{eq:def_eta}, we conclude\footnote{If $(X_{n})_{n\in\N}$ is $L^{r}$-bounded by a constant $c$ and $(Y_{\alpha})_{\alpha\in A}$ a family of positive, integer-valued random variables independent of $(X_{n})_{n\in\N}$, then $(X_{Y_{\alpha}})_{\alpha\in A}$ is also $L^{r}$-bounded by $c$, for
$$\Erw |X_{Y_{\alpha}}|^{r} = \sum_{n=1}^\infty \Erw (|X_{Y_{\alpha}}|^{r}|Y_{\alpha}=n)\,\Prob\{Y_{\alpha}=n\} =\sum_{n=1}^\infty \Erw|X_n|^{r}\,\Prob\{Y_{\alpha} = n\} \leq c^{r} \sum_{n=1}^\infty \Prob\{Y_{\alpha} = n\}= c^{r}.$$} that
\begin{equation}\label{eq:moments_bound}
\sup_{n\in\N} \Erw \left(\frac{\eta_{j}^{(n-1)}}{{L}_{1}(\eta_{j}^{(n)})}\right)^{r}\,<\,C_{r}
\quad\text{and}\quad
\sup_{n\in\N} \Erw \left|\frac{{L}_{1}(\eta_{j}^{(n)})-\eta_{j}^{(n-1)}}{\sqrt{\eta_{j}^{(n-1)}}}\right|^{r}\,<\,C_{r}.
\end{equation}
Based on these observations and by making use of the mean-value theorem for differentiable functions, we find
\begin{align*}
\Erw |\cX_{n,\,j}-\cX_{n-1,\,j}|\
&=\ \Erw \Big|{L}_{n}(\eta_{j}^{(n)})-{L}_{n-1}(\eta_{j}^{(n-1)})\Big|\\
&=\ \Erw \Big|{L}_{n-1}\Big({L}_{1}(\eta_{j}^{(n)})\Big)-{L}_{n-1}(\eta_{j}^{(n-1)})\Big|\\
&=\ \Erw \left({L}_{n-1}'(\theta_{n}(j))\Big|{L}_{1}(\eta_{j}^{(n)})-\eta_{j}^{(n-1)}\Big|\right)
\end{align*}
for some
$$ \theta_{n}(j)\in \left[{L}_{1}(\eta_{j}^{(n)})\wedge\eta_{j}^{(n-1)},\, {L}_{1}(\eta_{j}^{(n)})\vee \eta_{j}^{(n-1)}\right]. $$
For $n\ge 2$ and $x\in [1,\infty)$, we have
\begin{equation}\label{eq:L_derivative_est}
0 < {L}_{n-1}'(x)\le \frac{1}{x}
\end{equation}
and then by the Cauchy-Schwarz inequality,
\begin{equation*}
\Erw |\cX_{n,\,j}-\cX_{n-1,\,j}|\ \le\ \left\|\frac{\sqrt{\eta_{j}^{(n-1)}}}{\theta_{n}(j)}\right\|_{2}
\left\|\frac{{L}_{1}(\eta_{j}^{(n)})-\eta_{j}^{(n-1)}} {\sqrt{\eta_{j}^{(n-1)}}}\right\|_{2}
\end{equation*}
The second term on the right-hand side is bounded by an appeal to \eqref{eq:moments_bound}, while $\Erw\left(\eta_{j}^{(n-1)}\theta_{n}^{-2}(j)\right)$ can be further bounded from above by
\begin{align*}
\Erw&\left(\left(\eta_{j}^{(n-1)}\right)^{-1}\1_{\left\{{L}_{1}(\eta_{j}^{(n)})>\eta_{j}^{(n-1)}\right\}}
\ +\ \frac{\eta_{j}^{(n-1)}}{{L}^2_{1}(\eta_{j}^{(n)})}\1_{\left\{{L}_{1}(\eta_{j}^{(n)})\le \eta_{j}^{(n-1)}\right\}}\right)
\\
&\le\ \Erw\left(\eta_{j}^{(n-1)}\right)^{-1}\
+\ \Erw \left(\frac{1}{\eta_{j}^{(n-1)}}\left(\frac{\eta_{j}^{(n-1)}}{{L}_{1}(\eta_{j}^{(n)})}\right)^2\right)\\
&\le\ \Erw \left(\eta_{j}^{(n-1)}\right)^{-1}
\ +\ \left\|\frac{1}{\eta_{j}^{(n-1)}}\right\|_{2}\left\|\frac{\eta_{j}^{(n-1)}}{{L}_{1}(\eta_{j}^{(n)})}\right\|_4^{2}.
\end{align*}
Now the second factor in the second summand is bounded by an absolute constant; see again \eqref{eq:moments_bound}. By combining the previous estimates and using that $\eta_{j}^{(n-1)}\ge 1$, we finally arrive at
\begin{equation}\label{eq:E_X_{n}_j_X_{n}-1_j}
\Erw |\cX_{n,\,j}-\cX_{n-1,\,j}|\le C \left(\Erw \left(\frac{1}{\eta_{j}^{(n-1)}}\right)\right)^{1/4},
\quad n\ge 2,
\end{equation}
where $C$ is an absolute constant (not depending on $n$ and $j$). For $n=1$, this inequality is not valid because \eqref{eq:L_derivative_est} breaks down in this case. Instead, we then have $L_{0}'(x) = 1$ and
\begin{equation}\label{eq:E_X_{n}_j_X_{n}-1_j_special_case}
\Erw |\cX_{1,\,j}-\cX_{0,\,j}|\ =\ \Erw \left|{L}_{1}(\eta_{j}^{(1)})-j\right|\ \le\
\sqrt j \left\|\frac{{L}_{1}(\eta_{j}^{(1)})-j} {\sqrt{j}}\right\|_{2}\le\ C \sqrt j,
\end{equation}
byan appeal to the Cauchy-Schwarz inequality and then~\eqref{eq:moments_bound}. We will need~\eqref{eq:E_X_{n}_j_X_{n}-1_j_special_case} in the proof of Theorem~\ref{SLLN/CLT_S_{n}} below.

\vspace{.1cm}
Since $j\ge 2$, we have $\eta_{j}^{(n-1)}\ge 2$ for all $n\in\N$. Recalling from \eqref{eq:def_eta} that $\eta_{j}^{(n)} = \nu^{(n)}(\eta_{j}^{(n-1)})$ and applying Lemma~\ref{expectation_lemma} from the Appendix, we obtain that
\begin{align*}
\Erw\left(\eta_{j}^{(n)}\right)^{-1}\ &=\ \sum_{k=2}^{\infty} \Prob\{\eta_{j}^{(n-1)} = k\}\,\Erw\left(\nu^{(n)}(k)\right)^{-1}\\
&\le\ \frac{3}{4} \sum_{k=2}^{\infty} \frac{1}{k}\,\Prob\{\eta_{j}^{(n-1)} = k\}\
=\ \frac{3}{4}\,\Erw\left(\eta_{j}^{(n-1)}\right)^{-1}
\end{align*}
for all $n\in\N$. Using this inductively together with $\eta_j^{(0)} = j$, we get
$$ \Erw\left(\eta_{j}^{(n)}\right)^{-1}\ \le\ \left(\frac34\right)^{n} \frac 1j $$
for $n\in\N_{0}$ and $j\ge 2$. Therefore, recalling \eqref{eq:E_X_{n}_j_X_{n}-1_j},
\begin{align}
\begin{split}\label{eq:sum_expect_diff_X_{n}_j}
\sum_{n=2}^{\infty}\Erw |\cX_{n,\,j}-\cX_{n-1,\,j}|\ &\le\
C \sum_{n=2}^{\infty}\left(\Erw\left(\eta_{j}^{(n-1)}\right)^{-1}\right)^{1/4}\\
&\le\ \frac C{j^{1/4}} \sum_{n=1}^{\infty} \left(\frac34\right)^{n/4}\ \le\ \frac{C'}{j^{1/4}},
\end{split}
\end{align}
where $C'$ is an absolute constant. This completes the proof of~\eqref{X_bv}.

\vspace{.1cm}
The random mapping $\psi$ defined by \eqref{psi_mapping_def} is almost surely continuous with respect to the product topology on $\R^{\N}$, hence the limit vector satisfies \eqref{fp_limit_S}. This completes the proof of Theorem \ref{main1}.\qed

\subsection{... of Theorems \ref{SLLN/CLT_S_{n}}, \ref{limit_poisson}, \ref{main1_time} and Proposition \ref{prop:no_atoms}}

\begin{proof}[of Theorem~\ref{SLLN/CLT_S_{n}}]
We first prove that $k^{-1}S_{k}^{*}\to 1$ a.s. In the proof of Theorem~\ref{main1}, we have shown that, for each fixed $k\ge 2$, the sequence of random variables $(\cX_{n,k})_{n\in\N_{0}}$ converges a.s.\ to a random variable here denoted by $\cX_{k}^{*}$. Since $(\cX_k^{*})_{k\in \N}$ has the same distribution as $(S_k^{*})_{k\in\N}$, it suffices to prove that 
\begin{equation}\label{eq:LLN_restated_{1}}
k^{-1}\cX_k^{*}\ \to\ 1\quad\text{a.s.}
\end{equation}
Recall from~\eqref{eq:def_X} that $\cX_{1,k}= {L}_{1} (\nu^{(1)}(k))$.  By the strong law of large numbers for record times, see~\cite{Renyi:62} or~\eqref{eq:lln_clt_{n}u_{n}} in the Appendix, we have
\begin{equation}\label{eq:LLN_record_times}
k^{-1}{L}_{1}(\nu^{(1)}(k))\ \to\ 1\quad\text{a.s.}
\end{equation}
Observe that
\begin{align}
\label{eq:E_X_minus_L_{n}u}
\Erw\left|\cX^{*}_k -  {L}_{1} (\nu^{(1)}(k))\right|\ &=\ \Erw\left|\sum_{n=2}^{\infty}(\cX_{n,k} - \cX_{n-1,k})\right|\ \le\ \sum_{n=2}^{\infty} \Erw\left|\cX_{n,k} - \cX_{n-1,k}\right|\ \le\
\frac{C'}{k^{1/4}},
\end{align}
where we have utilized~\eqref{eq:sum_expect_diff_X_{n}_j} for the last step. By Markov's inequality it follows that, for every $\eps>0$,
$$ \Prob\left\{\left|\frac{\cX_k^{*}}{k}-\frac{{L}_{1}(\nu^{(1)}(k))}{k}\right|>\eps \right\}
\ \le\ \frac{1}{\eps}\,\Erw\left|\frac{\cX_k^{*}- {L}_{1} (\nu^{(1)}(k))}{k}\right|\
\le\\ \frac{C'}{\eps k^{5/4}}, $$
and since the last term is summable over $k$, the Borel-Cantelli lemma together with~\eqref{eq:LLN_record_times} implies~\eqref{eq:LLN_restated_{1}}.

\vspace{.1cm}
To derive the asymptotic formula for the expectation, namely $\Erw S_k^{*} \sim k$, observe that
\begin{align}
\begin{split}\label{eq:E_X_minus_k}
\Erw\left|\cX_{k}^{*}- k\right|\ &=\ \Erw\left|\sum_{n=1}^{\infty}(\cX_{n,k} - \cX_{n-1,k})\right|\\
&\le\ \sum_{n=1}^{\infty} \Erw\left|\cX_{n,k} - \cX_{n-1,k}\right|\ \le\
C\sqrt{k} + \frac{C'}{k^{1/4}}\ \le\ C''\sqrt k,
\end{split}
\end{align}
where \eqref{eq:E_X_{n}_j_X_{n}-1_j_special_case} and~\eqref{eq:sum_expect_diff_X_{n}_j} have been used.

\vspace{.1cm}
Finally turning to the central limit theorem, it suffices again to show the assertion for the sequence $(\cX_{k}^{*})_{k\in\N_0}$, thus
\begin{equation}\label{eq:CLT_restate}
k^{-1/2}\big(\cX_k^{*}-k\big)\ \todistr\ \text{N}(0,1).
\end{equation}
The central limit theorem for record times, see \cite[page 63]{Nevzorov:01} or~\eqref{eq:lln_clt_{n}u_{n}} in the Appendix, states that
\begin{equation}\label{eq:CLT_proof1}
k^{-1/2}\big(L_{1} (\nu^{(1)}(k))-k\big)\ \todistr\ \text{N}(0,1),
\end{equation}
and from \eqref{eq:E_X_minus_L_{n}u}, we know that
\begin{equation}\label{eq:CLT_proof2}
\Erw\left|\frac{\cX_k^{*}-k}{\sqrt k}-\frac{L_{1}(\nu^{(1)}(k))-k}{\sqrt k}\right|\ =\
\frac{1}{\sqrt{k}}\Erw\left|\cX_{k}^{*}-L_{1}(\nu^{(1)}(k))\right|\ \le\ \frac{C'}{k^{3/4}}.
\end{equation}
The combination of \eqref{eq:CLT_proof1} and~\eqref{eq:CLT_proof2} gives \eqref{eq:CLT_restate}.\qed
\end{proof}

\begin{proof}[of Theorem~\ref{limit_poisson}]
Fix $m\in\N$. Note that $e^{\cE_{1}}, \ldots, e^{\cE_{m}}$ are Pareto-distributed with survival function $\Prob\{e^{\cE_{i}}>u\} = 1/u$ for $u>1$. It is known, see~\cite[page 64]{Nevzorov:01} and~\cite[Theorem 5]{Shorrock:72},  that
$$ \left(\frac {\nu^{(1)}(k+l)}{\nu^{(1)}(k+l-1)}\right)_{l=1,\ldots,m}\ \todistr\ (e^{\cE_l})_{l=1,\ldots,m}\quad\text{as }k\to\infty. $$
By the continuous mapping theorem, it follows that, as $k\to\infty$,
\begin{align*}
\Big(L_{1} (\nu^{(1)}(k+l)) &- L_{1} (\nu^{(1)}(k+l-1)) \Big)_{l=1,\ldots,m}\\
&=\ \left(\log  \frac {\nu^{(1)}(k+l)}{\nu^{(1)}(k+l-1)}\right)_{l=1,\ldots,m}
\ \todistr\ (\cE_l)_{l=1,\ldots,m}.
\end{align*}
Moreover, we know from~\eqref{eq:E_X_minus_L_{n}u} that
$$ \Erw\left|(\cX_{k+l}^{*}-\cX_{k+l-1}^{*}) - (L_{1} (\nu^{(1)}(k+l)) - L_{1} (\nu^{(1)}(k+l-1))) \right|\ \le\ \frac {2C'}{k^{1/4}}. $$
Therefore, as $k\to\infty$,
$$ (\cX_{k+l}^{*}-\cX_{k+l-1}^{*})_{l=1,\ldots,m}\ \todistr\ (\cE_l)_{l=1,\ldots,m}, $$
which completes the proof of \eqref{eq:conv S_j^*} because $(\cX_{k+l}^{*}-\cX_{k+l-1}^{*})_{l=1,\ldots,m}$ has the same law as $(S_{k+l}^{*}-S_{k+l-1}^{*})_{l=1,\ldots,m}$.\qed
\end{proof}

\begin{proof}[of Prop.~\ref{prop:no_atoms}]
We must show that $S_{k}^{*}$ has no atoms for all $k\ge 2$. Assuming the contrary, let $a$ be an atom of $S_k^{*}$ having maximal mass, thus $p:=\Prob\{S_{k}^{*}=a\}>0$ and $\Prob\{S_{k}^{*}=b\}\le p$ for all $b\in\R$. The stochastic fixed-point equation \eqref{fp_limit_S} then gives
\begin{align*}
p\ &=\ \Prob\{S_k^{*}=a\}\ =\ \Prob\{ S_{\nu(k)}^{*}=E_{1}(a)\}\\
&= \Prob\{\nu(k)= k\}\,\Prob\{ S_k^{*}=E_{1}(a)\}\ +\ \sum_{l=k+1}^{\infty}\,\Prob\{\nu(k)=l\}\,\Prob\{ S_l^{*}=E_{1}(a)\}.
\end{align*}
Since $\Prob\{ S_k^{*} =  E_{1}(a)\}\le p$ and $\Prob\{\nu(k)=k\} = 1/k!<1$, we infer that $ S_{l}^{*}$, for at least one $l\ge k+1$, satisfies $\Prob\{ S_l^{*}= E_{1}(a)\}\ge p$ and thus has also an atom of size at least $p$. Repeating this argument indefinitely, we obtain a sequence $k < k_{1} < k_{2} < \ldots$ such that each $S_{k_{i}}^{*}$ has an atom of size at least $p$. As a consequence, the random variables
$$ W_{i}\ :=\ k_{i}^{-1/2}\big(S_{k_{i}}^{*}-k_{i}\big),\quad i\in\N, $$
have this property as well. On the other hand, we know by Theorem~\ref{SLLN/CLT_S_{n}} that $W_{i}$ converges in distribution to a standard normal random variable. By Lemma~\ref{lem:no_atoms_weak_conv} below, this is a contradiction.

\vspace{.1cm}
Next, we must show \eqref{eq:all values attained} or, equivalently,
$$ p\,:=\,\Prob\{\cX_{2}^{*} \in [\alpha_{2},\beta_{2}], \ldots,\cX_{m}^{*} \in [\alpha_{m},\beta_{m}]\}\,>\,0 $$
for any $1<\alpha_{2}<\beta_{2}<\ldots<\alpha_{m}<\beta_{m}$, for $(S_{2}^{*}, \ldots, S_m^{*})\eqdist(\cX_{2}^{*}, \ldots, \cX_m^{*})$. Fix $k\in\N$ so large that, for each $i=2,\ldots,m$, the interval $(E_k(\alpha_{i}),E_k(\beta_{i}))$ contains an integer $t_{i}>\max\{i, (2C''m)^2\}$ (with $C''$ as in~\eqref{eq:E_X_minus_k}) such that
\begin{equation}\label{eq:wspom_t_{i}}
E_k(\alpha_{i})<\frac{t_{i}}{2} < t_{i} <\frac{3t_{i}}{2} < E_k(\beta_{i}).
\end{equation}
The existence of such $t_{i}$ follows from  $\lim_{k\to\infty}E_k(\beta_{i})/E_k(\alpha_{i})=\infty$. One can take, for instance, $t_{i}:=[E_k(\beta_{i})/2]$. Recalling that $\cX_j^{*}$ is the almost sure limit of $\cX_{n,j}=L_{n}(\eta_{j}^{(n)})$ as $n\to\infty$, where $\eta_{j}^{(n)}=\nu^{(n\darr 1)}(j)$, and using the independence of $\nu^{(1)},\ldots,\nu^{(n)}$, we obtain
\begin{align*}
p\ > ~&\Prob\{\eta_{2}^{(k)} = t_{2}, \ldots, \eta_m^{(k)} = t_m\}\\
&\times\ \Prob\left\{ \lim_{n\to\infty} L_{n} (\nu^{(n\darr k+1)}(t_{i})) \in [\alpha_{i},\beta_{i}] \text{ for all } i=2,\ldots,m\right\}.
\end{align*}
It suffices to show that both probabilities on the right-hand side, say $p'$ and $p''$, are strictly positive. As for $p'$, this follows from
$$ p'\ \ge\ \Prob\{\nu^{(1)}(i)=\ldots=\nu^{(k-1)}(i)=i,\,\nu^{(k)} (i) = t_{i}\text{ for }i=2,\ldots,m\}\ >\ 0. $$
Regarding $p''$, we infer with the help of \eqref{eq:wspom_t_{i}}
\begin{align*}
p''\ &=\ \Prob\{ \cX_{t_{i}}^{*}  \in [E_{k}(\alpha_{i}),E_{k}(\beta_{i})] \text{ for } i=2,\ldots,m\}\\
&\ge\ \Prob\left\{ \frac{t_{i}}2 \le  \cX_{t_{i}}^{*} \leq \frac{3t_{i}}{2} \text{ for all } i=2,\ldots,m\right\},
\end{align*}
and Markov's inequality in combination with \eqref{eq:E_X_minus_k} provides us with
$$ \Prob \left\{\left|\cX_{t_{i}}^{*}- t_{i}\right|\ >\ \frac{t_{i}}2 \right\}\ \le\
\frac{2}{t_{i}}\,\Erw \left|\cX_{t_{i}}^{*}- t_{i}\right|\ \le\ \frac{2C''}{\sqrt{t_{i}}}\ <\ \frac{1}{m}. $$
Consequently,
$$ p''\ >\ 1-\sum_{i=2}^{m} \Prob \left\{\left|\cX_{t_{i}}^{*}- t_{i}\right|\ >\ \frac{t_{i}}2 \right\}\ >\  1-\frac{m-1}{m}\ >\ 0, $$
which completes the proof.\qed
\end{proof}

\begin{Lemma}\label{lem:no_atoms_weak_conv}
Let $W_{1},W_{2},\ldots$ be random variables such that $\Prob\{W_{n}=a_{n}\}\ge p>0$ for some $a_{1},a_{2},\ldots\in\R$ and all $n\in\N$. Then $W_{n}\todistr W$ implies that $W$ has also an atom.
\end{Lemma}

\begin{proof}
Suppose that $W_{n}$ converges in distribution to a random variable $W$ with a continuous law. If $\limsup_{n\to\infty}a_{n}=+\infty$, then
$$ \Prob\{W\le x\}\ =\ \lim_{n\to\infty}\Prob\{W_{n}\le x\}\ \le\ \limsup_{n\to\infty}\Prob\{W_{n}<a_{n}\}\ \le\ 1-p, $$
for all $x\in\R$ which is impossible, for $\lim_{x\to+\infty} \Prob\{W\le x\} = 1$. By a similar argument, we can rule out $\liminf_{n\to\infty}a_{n}=-\infty$. So it remains to consider the case when the sequence $(a_{n})$ is bounded, w.l.o.g. (after passing to a subsequence) convergent to some $a\in\R$. Then
$$ \Prob\{W\in [a-\eps,a+\eps]\}\ =\ \lim_{n\to\infty}\Prob\{W_{n}\in [a-\eps, a+\eps]\}\ \ge\ p, $$
for any $\eps>0$, which is again impossible, for $\lim_{\eps\downarrow 0} \Prob\{W\in [a-\eps,a+\eps]\} =0$ by the continuity of the law of $W$.\qed
\end{proof}

\begin{proof}[of Theorem \ref{main1_time}]
Recall that $S^{(n)}=(S^{(n)}_j)_{j\in\N}$. We must show that
\begin{equation}\label{eq:main1_time_statement}
(L_{n+k} (S^{(n+k)}))_{k\in \{-p,\ldots,p\}}\ \todistr\ (V_k)_{k\in\{-p,\ldots,p\}}
\end{equation}
for every $p\in\N$, where $L_{n+k}$ is applied coordinate-wise and we recall that $V_k = (S_{k,j}^{*})_{j\in\N}$. We regard both sides in \eqref{eq:main1_time_statement} as random vectors with components from $\R^{\N}$, the latter space being endowed with the product topology. Recall that $(\psi^{(j)})_{j\in\Z}$ is a family of independent copies of the mapping $\psi$ and assume that $S^{(n-p)}$ is independent of this family.   By definition of $\psi$, see~\eqref{psi_mapping_def},
$$ \left(L_{p+k}(S^{(n+k)})\right)_{k\in\{-p,\ldots,p\}}\ \eqdist\ \left(\psi^{(k-1\darr -p)}(S^{(n-p)})\right)_{k\in\{-p,\ldots,p\}}, $$
where we have used that $\psi^{(j)}$ commutes with $L_{1}$ when applied coordinate-wise. Applying the map $L_{n-p}$ to both sides, we obtain
\begin{align*}
\left(L_{n+k}(S^{(n+k)})\right)_{k\in \{-p,\ldots,p\}}\
\eqdist\ \left(\psi^{(k-1 \darr -p)} (L_{n-p}(S^{(n-p)}))\right)_{k\in \{-p,\ldots,p\}},
\end{align*}
and we know from Theorem~\ref{main1} that $L_{n-p}(S^{(n-p)})\idist V_{-p}$ as $n\to\infty$.
By the continuous mapping theorem,
$$ \left(L_{n+k}(S^{(n+k)})\right)_{k\in \{-p,\ldots,p\}}\ \todistr\ (\psi^{(k-1 \darr -p)} (V_{-p}))_{k\in \{-p,\ldots,p\}}\ =\ (V_{k})_{k\in \{-p,\ldots,p\}}, $$
thus completing the proof of \eqref{eq:main1_time_statement} and Theorem \ref{main1_time}.\qed
\end{proof}

\subsection{... of Theorems~\ref{main2} and \ref{main2_time}}

\begin{proof}[of Theorem \ref{main2}]
Our argument is based on the simple duality relation
\begin{equation}\label{duality_S_N}
\{N^{(n)}_{M}\ge k\}\ =\ \{S_{k}^{(n)}\le M\},\quad n,k,M\in\N.
\end{equation}
Fix $m\in\N$, $1\le \rho_{1}<\ldots<\rho_m$, $k_{1},\ldots,k_m\in\{2,3,\ldots\}$ and write
\begin{align*}
\Prob\{N_{[{E}_{n}(\rho_{i})]}^{(n)}\ge k_{i},\,i=1,\ldots,m\}
&=\ \Prob\{S_{k_{i}}^{(n)} \le [{E}_{n}(\rho_{i})],\,i=1,\ldots,m\}\\
&=\ \Prob\{S_{k_{i}}^{(n)} \le {E}_{n}(\rho_{i}),\,i=1,\ldots,m\}.
\end{align*}
Applying the monotone function ${L}_{n}$ to both sides, we obtain
$$ \Prob\{N_{[{E}_{n}(\rho_{i})]}^{(n)}\ge k_{i},\,i=1,\ldots,m\}\ =\
\Prob\{L_{n}(S_{k_{i}}^{(n)})\le\rho_{i},\,i=1,\ldots,m\}. $$
By Theorem~\ref{main1}, we know that
$$ \left(L_{n}(S^{(n)}_{k_{1}}), \ldots,L_{n}(S^{(n)}_{k_m})\right)\ \todistr\ (S_{k_{1}}^{*},\ldots, S_{k_m}^{*}), $$
and $S_{k_{1}}^{*},\ldots, S_{k_m}^{*}$ have no atoms by Prop.~\ref{prop:no_atoms}. A combination of these facts provides us with
\begin{align*}
\lim_{n\to\infty}\Prob\left\{N_{[{E}_{n}(\rho_{i})]}^{(n)}\ge k_{i},\,i=1,\ldots,m\right\}\ &=\
\Prob\{{S}_{k_{i}}^{*}\le \rho_{i},\,i=1,\ldots,m\}\\
&=\ \Prob\{{N}^{*}(\rho_{i})\ge k_{i},\,i=1,\ldots,m\},
\end{align*}
where \eqref{N_ast_def} should be recalled for the last equality.  This proves \eqref{eq:N_{n}_conv_fdd}.  To prove the stochastic fixed point equation~\eqref{eq:fixed_point_N_star}, observe that by \eqref{fp_limit_S},
\begin{align*}
{N}^{*}(1+\log \rho)\ &=\ \#\{k\in\N\colon S^{*}_k\le {L}_{1}(\rho)\}\\
&\eqdist\ \#\{k\in\N\colon {L}_{1}(S^{*}_{\nu(k)})\le {L}_{1}(\rho)\}\\
&=\ \#\{k\in\N\colon S^{*}_{\nu(k)}\le \rho\}\ =\ K({N}^{*}(\rho)),
\end{align*}
and a similar identity holds for the finite-dimensional distributions of the processes. \qed
\end{proof}

\begin{proof}[of Theorem \ref{main2_time}] We restrict ourselves to a proof of
\begin{equation}\label{eq:proof_main2_time}
\lim_{n\to\infty}\Prob\{N^{(n+k_{i})}_{[E_{n}(\rho)]}\ge l_{i},\,i=1,\ldots,m\}\ =\ \Prob\{N^{*}_{k_{i}}(\rho)\ge l_{i},\,i=1,\ldots,m\}
\end{equation}
for arbitrary $m\in\N$, $k_{1},\ldots,k_m\in\Z$, $l_{1},\ldots,l_m\in\{2,3,\ldots\}$ and $\rho>1$. The general case when $\rho$ may also vary can be treated by exactly the same arguments. We have, using \eqref{duality_S_N},
\begin{align*}
\Prob\{N^{(n+k_{i})}_{[E_{n}(\rho)]}\ge l_{i},\,i=1,\ldots,m\}\ &=\ \Prob\{S^{(n+k_{i})}_{l_{i}}\le E_{n}(\rho),\,i=1,\ldots,m\}\\
&=\ \Prob\{L_{n+k_{i}}(S^{(n+k_{i})}_{l_{i}})\le L_{k_{i}}(\rho),\,i=1,\ldots,m\}.
\end{align*}
By Theorem \ref{main1_time}, the last probability on the right-hand side converges to
$$
\Prob\{S^{*}_{k_{i},l_{i}}\le L_{k_{i}}(\rho),\,i=1,\ldots,m\}\ =\ \Prob\{N^{*}_{k_{i}}(\rho)\ge l_{i}, \,i=1,\ldots,m\}$$
as $n\to\infty$, thus proving \eqref{eq:proof_main2_time}.\qed
\end{proof}

\subsection{... of Theorems \ref{mainT0}, \ref{mainT}, \ref{LLN_for_T} and Prop.~\ref{prop:no_further_limits}\label{subsec:proof_mainT}}

We first give a proof of Theorem \ref{mainT} because the proof of Theorem \ref{mainT0} will draw on this result (and its Corollary \ref{mainT_tilde}).

\begin{proof}[of Theorem \ref{mainT}]
We have divided the proof into five steps.

\vspace{.1cm}\noindent
\textsc{Step 1.} We first prove \eqref{eq:T_fdd_conv}, i.e.
\begin{equation*}
(T([E_{n}(\rho)])-n)_{\rho>1}\todistrfd (T^{*}(\rho))_{\rho>1}.
\end{equation*}
Fix $m\in\N$, $1<\rho_{1}<\ldots <\rho_m$ and $k_{1},\ldots,k_m\in\Z$ and note that for large $n$
\begin{align*}
\Prob\{T([E_{n}(\rho_{i})])-n\le k_{i},\,i=1,\ldots,m\}\ &=\ \Prob\{S^{(n+k_{i})}_{2} > E_{n}(\rho_{i}),\,i=1,\ldots,m\}\\
&=\ \Prob\{L_{n+k_{i}}(S^{(n+k_{i})}_{2}) > L_{k_{i}}(\rho_{i}),\,i=1,\ldots,m\},
\end{align*}
where \eqref{eq:group_property} has been used for the last step and $L_{-k}=E_{k}$ for $k\le 0$ should be recalled. By Theorem~\ref{main1_time}, we know that
$$ (L_{n+k_{1}}(S^{(n+k_{1})}_{2}), \ldots,L_{n+k_m}(S^{(n+k_m)}_{2}))\ \todistr\
(S_{k_{1},2}^{*},\ldots, S_{k_{m},2}^{*}), $$
and the random variables ${S}^{*}_{k_{1},2},\ldots,{S}^{*}_{k_{m},2}$ (by definition, of the same law as ${S}^{*}_{2}$) have no atoms by Prop.~\ref{prop:no_atoms}. By combining these facts,
$$ \lim_{n\to\infty} \Prob\{T([E_{n}(\rho_{i})])-n\le k_{i},\,i=1,\ldots,m\}\ =\
\Prob\{S_{k_{i},2}^{*} >L_{k_{i}}(\rho_{i}),\,i=1,\ldots,m\} $$
follows, and this completes the proof of~\eqref{eq:T_fdd_conv} if we put, for $\rho>1$,
\begin{equation}\label{T_explicit}
T^{*}(\rho)\ :=\ \inf\{k\in\Z:{S}_{k,2}^{*}>{L}_{k}(\rho)\}\ =\ \inf\{k\in\Z:E_{k}({S}_{k,2}^{*})>\rho\}.
\end{equation}
The infimum is indeed well-defined because ${S}_{k,2}^{*}\to\infty$ as $k\to\infty$ by Theorem~\ref{SLLN/CLT_S_{n}}, and $E_{k}({S}_{k,2}^{*})$ is nondecreasing in $k$:
$$ E_{k+1}({S}_{k+1,2}^{*})\ =\ E_{k+1}(L_{1}({S}_{k,\nu(2)}^{*})\ =\ E_{k}({S}_{k,\nu(2)}^{*})\ \ge\  E_{k}({S}_{k,2}^{*}), $$
in view of $\nu(2)\ge 2$.

\vspace*{.2cm}\noindent
\textsc{Step 2.} That $(T^{*}(\rho))_{\rho>1}$ and $(T^{*}(e^{\rho-1})-1)_{\rho>1}$ have the same finite-dimensional distributions, follows by a double use of \eqref{eq:T_fdd_conv}, viz.
\begin{align*}
(T^{*}(\rho))_{\rho>1}\ \overset{f.d.d.}{\longleftarrow}&\ (T([E_{n+1}(\rho)])-(n+1))_{\rho>1}\\
=&\ (T([E_{n}(e^{\rho-1})])-n-1)_{\rho>1}\ \overset{f.d.d.}{\longrightarrow}\
(T^{*}(e^{\rho-1})-1)_{\rho>1}.
\end{align*}

\vspace*{.2cm}\noindent
\textsc{Step 3.}
It is clear from~\eqref{T_explicit} that the sample paths of $(T^{*}(\rho))_{\rho>1}$ are nondecreasing and c\`adl\`ag. To prove the stochastic continuity, observe that
\begin{align*}
\Prob\{T^{*} (\rho+\eps) &- T^{*}(\rho-\eps) \ge 1\}\\
&=\ \lim_{n\to\infty} \Prob\{T([E_{n}(\rho+\eps)]) - T([E_{n}(\rho-\eps)]) \ge 1\}\\
&=\ \lim_{n\to\infty} \Prob\{E_{n}(\rho-\eps) < S_k^{(n)} \le E_{n}(\rho+\eps)\text{ for some }k\in\N\}\\
&\le\ \lim_{n\to\infty} \sum_{k=1}^{\infty}\Prob\{E_{n}(\rho-\eps) < S_k^{(n)} \le E_{n}(\rho+\eps)\}.
\end{align*}
Recalling that $L_{n}(S_k^{(n)})$ has the same distribution as $\cX_{n,k}$, we obtain
\begin{equation}\label{eq:stoch cont estimate}
\Prob\{T^{*} (\rho+\eps) -T^{*}(\rho-\eps) \ge 1\}\ \le\
\lim_{n\to\infty} \sum_{k=1}^{\infty}\Prob\{\rho-\eps < \cX_{n,k} \le \rho+\eps\}.
\end{equation}
For every  fixed $k\in\N$, we have
$$ \lim_{\eps\downarrow 0} \lim_{n\to\infty} \Prob\{\rho-\eps \le \cX_{n,k} \le \rho+\eps\}\ =\  \lim_{\eps\downarrow 0} \Prob\{\rho-\eps \le \cX_{k}^{*} \le \rho+\eps\}\ =\ 0 $$
because $\cX_{k}^{*}\eqdist S_k^{*}$ has no atoms (Prop.~\ref{prop:no_atoms}). The stochastic continuity will now follow from \eqref{eq:stoch cont estimate} via the dominated convergence theorem once the estimate
\begin{equation}\label{eq:mathcal_X_small}
\Prob\{\cX_{n,k} \le a\} \le \frac{b}{k^{5/4}}
\end{equation}
has been established for all $a$ and sufficiently large $n,k$, and some $b>0$. To this end, note that, for large $k$,
$$ \Prob\{\cX_{n,k} \le a\}\ \le\ \Prob\left\{|\cX_{n,k} - L_{1}(\nu^{(1)}(k))| > \frac{k}{3}\right\}
\,+\,\Prob\left\{|L_{1}(\nu^{(1)}(k)) - k| > \frac{k}{3}\right\}. $$
Recalling~\eqref{eq:E_X_minus_L_{n}u}, we infer with the help of Markov's inequality
\begin{align*}
\Prob\left\{|\cX_{n,k} - L_{1}(\nu^{(1)}(k))| > \frac{k}{3}\right\}\ &\le\ \frac{3}{k}\Erw|\cX_{n,k} - L_{1}(\nu^{(1)}(k))|\\
&\le\ \frac{3}{k}\sum_{l=2}^{n}\Erw|\cX_{l,k} - \cX_{l-1,k}|\ \le\ \frac{3C'}{k^{5/4}},
\end{align*}
and for the second term in a similar manner that
$$ \Prob\left\{|L_{1}(\nu^{(1)}(k)) - k| > \frac{k}{3}\right\}\ \le\ \frac{3^{4}}{k^2}\,\Erw\left(\frac{L_{1}(\nu^{(1)}(k)) - k}{\sqrt k}\right)^4\ \le\ \frac{3^{4}}{k^2}C, $$
where Proposition~\ref{nu_{n}_prop} in the Appendix has been utilized for the last estimate.

\vspace*{.2cm}\noindent
\textsc{Step 4.}
We next prove that $\Prob\{T^{*}(\rho)=k\} >0$ for all $k\in\Z$. Note that
$$ \Prob\{T^{*}(\rho)=k\}\ =\ \lim_{n\to\infty} \Prob\{T([E_{n}(\rho)])=n+k\}. $$
If exactly one person among $2,\ldots, [E_{n}(\rho)]$ survives at time $n+k-1$ and if this person leaves at time $n+k$, then the event $\{T([E_{n}(\rho)]) = n+k\}$ occurs. Consequently,
\begin{align*}
\Prob\{T([E_{n}(\rho)])=n+k\}\ &\ge\ \frac{1}{2}\,\Prob\left\{S_{2}^{(n+k-1)}\le E_{n}(\rho)<S_{3}^{(n+k-1)}\right\}\\
&=\ \frac{1}{2}\,\Prob\left\{L_{n+k-1}(S_{2}^{(n+k-1)}) \le L_{k-1}(\rho) < L_{n+k-1}(S_{3}^{(n+k-1)})\right\}.
\end{align*}
Letting $n\to\infty$ and using Theorem~\ref{main1} and Prop.~\ref{prop:no_atoms}, we obtain
$$ \Prob\{T^{*}(\rho) =k\}\,=\,\lim_{n\to\infty} \Prob\{T([E_{n}(\rho)])=n+k\}\,\ge\,\frac{1}{2}\,\Prob\left\{S_{2}^{*}\le L_{k-1}(\rho)\le S_{3}^{*}\right\}, $$
and the rightmost probability is strictly positive by Prop.~\ref{prop:no_atoms}.

\vspace*{.2cm}\noindent
\textsc{Step 5.} We must finally prove that $T^{*}(\rho_{1}),\,T^{*}(\rho_{2})$ have distinct laws for any $1<\rho_{1} < \rho_{2}$. In fact, we even show \emph{strict} stochastic domination, i.e.
$$ \Prob\{T^{*}(\rho_{1})\le k\} > \Prob\{T^{*}(\rho_{2})\le k\} $$
for all $k\in\Z$. By~\eqref{T_explicit}, we have
\begin{align*}
\Prob\{T^{*}(\rho_{1})\le k\} - \Prob\{T^{*}(\rho_{2})\le k\}\ &=\ \Prob\{S_{2}^{*}>L_k(\rho_{1})\} - \Prob\{S_{2}^{*}>L_k(\rho_{2})\}\\
&=\ \Prob\{L_k(\rho_{1}) < S_{2}^{*}\le L_k(\rho_{2})\},
\end{align*}
and the last probability is strictly positive by Prop.~\ref{prop:no_atoms}.\qed
\end{proof}

\begin{proof}[of Theorem~\ref{mainT0}]
We prove that the sequence $(T(M)- \log^* M)_{M\in\N}$ is bounded in $L^{r}$ for every $r > 0$ and thus, in particular, tight. Fix $M\in\{2,3,\ldots\}$ and pick $n\in\N$ such that $\widetilde E_{n}(0)\le  M < \widetilde E_{n+1}(0)$. Note that $\log^*M = n$. For every $z\in\Z$, we have
\begin{align*}
\Prob[T([\widetilde E_{n}(0)])-n\ge z]\ &\le\ \Prob[T(M) - \log^* M \ge z]\ \le\ \Prob[T([\widetilde E_{n+1}(0)])-(n+1) \ge  z-1],
\end{align*}
whence it suffices to show that, for all $z\in\N$ and $n$ sufficiently large,
\begin{align}
&\Prob\{T([\widetilde E_{n}(0)])\ge n+z\}\ \le\ \frac{C_{1}}{2^{z}},\label{eq:tight_wspom_{1}}
\shortintertext{and}
&\Prob\{T([\widetilde E_{n}(0)])\le n - z\}\ \le\ \frac{C_{2}}{\widetilde E_z(0)}\label{eq:tight_wspom_{2}}.
\end{align}

\noindent
\textsc{Proof of~\eqref{eq:tight_wspom_{1}}.}  We have
\begin{align*}
\Prob\{T([\widetilde E_{n}(0)])\ge n+z\}\ &=\ \sum_{k=2}^{\infty}\Prob\left\{N_{[\widetilde E_{n}(0)]}^{(n)}=k\right\}\,\Prob\{N_{k}^{(z)}\ge 2\}.
\end{align*}
Use the crude estimate $\Prob\{N^{(z)}_{k}\ge 2\}\le k2^{-z}$  to obtain
$$ \Prob\{T([\widetilde E_{n}(0)])\ge n+z\}\ \le\ 2^{-z}\,\Erw N^{(n)}_{[\widetilde E_{n}(0)]}. $$
It therefore remains to show that $\sup_{n\in\N}\Erw N^{(n)}_{[\widetilde E_{n}(0)]}<\infty$. Observe that
$$ \Erw N^{(n)}_{[\widetilde E_{n}(0)]}\ =\ \Erw \left(\sum_{k=1}^{\infty} \1_{\{S_{k}^{(n)}\le\widetilde E_{n}(0)\}}\right)\ =\ \sum_{k=1}^{\infty} \Prob\{\mathcal{X}_{n,k} \le  L_{n} (\widetilde E_{n}(0))\}. $$
Finally, $\lim_{n\to\infty} L_{n} (\widetilde E_{n}(0))=f(0)$ and \eqref{eq:mathcal_X_small} imply
$$ \Prob\{\mathcal{X}_{n,k}\le L_{n} (\widetilde E_{n}(0))\}\ \le\ \Prob\{\mathcal{X}_{n,k}\le f(0)+1\}\ \le\ bk^{-5/4} $$
for sufficiently large $n$ and $k$, and this is enough for \eqref{eq:tight_wspom_{1}}.

\vspace*{.2cm}
\noindent
\textsc{Proof of~\eqref{eq:tight_wspom_{2}}.}
We have
\begin{align*}
\Prob\{T([\widetilde E_{n}(0)])\le n-z\}\ &=\ \Prob\{S_{2}^{(n-z)}>\widetilde E_{n}(0)\}\ =\ \Prob\{L_{n-z} (S_{2}^{(n-z)})>L_{n-z}(\widetilde E_{n}(0))\}\\
&\le\ \Prob\{L_{n-z} (S_{2}^{(n-z)})>\widetilde L_{n-z}(\widetilde E_{n}(0))\}\ =\ \Prob\{L_{n-z} (S_{2}^{(n-z)})>\widetilde E_{z}(0)\}
\end{align*}
and know that $\Erw L_{n-z}(S_{2}^{(n-z)})=\Erw\mathcal{X}_{n-z,2}$. But the sequence $(\Erw \mathcal{X}_{n,2})_{n\ge 1}$ is bounded, for it has bounded variation (see \eqref{eq:sum_expect_diff_X_{n}_j}). Hence, \eqref{eq:tight_wspom_{2}} follows with the help of Markov's inequality.

\vspace*{.2cm}
To prove that $T(M) - \log^*M$ does not converge in distribution, we argue as follows. For arbitrary $0<\rho_{1}<\rho_{2}<1$ and sufficiently large $n$,
$$ \widetilde{E}_{n}(0)\,<\,[\widetilde{E}_{n}(\rho_{1})]\,\le\,\widetilde{E}_{n}(\rho_{1})\,<\,[\widetilde{E}_{n}(\rho_{2})]\,\le\,\widetilde{E}_{n}(\rho_{2})\,<\,\widetilde{E}_{n+1}(0) $$
holds true, thus
$$ \log^* [\widetilde E_{n}(\rho_{i})]\,=\,n\quad\text{for }i=1,2. $$
From Corollary~\ref{mainT_tilde}, we know that, for $i=1,2$,
$$ T([\widetilde E_{n}(\rho_{i})]) - \log^* [\widetilde E_{n}(\rho_{i})]\ =\ T([\widetilde E_{n}(\rho_{i})])-n\ \todistr\ T^{*}(f(\rho_{i})). $$
Since $T^{*}(f(\rho_{1}))$ and $T^{*}(f(\rho_{2}))$ have distinct distributions by Theorem~\ref{mainT} (recall also that $f$ is strictly increasing), the claim is proved.\qed
\end{proof}

\begin{proof}[of Proposition~\ref{prop:no_further_limits}]
For any sufficiently large $n\in\N$, we can find $k(n)\in\N$ such that $E_{k(n)}(2)\le a_{n}<E_{k(n)+1}(2)$. Put $\rho_{n}:=L_{k(n)}(a_{n}) \in [2, e)$. Possibly after passing to a subsequence, we may assume that $\lim_{n\to\infty} \rho_{n} = \rho \in [2, e]$ and that $k(n)$ is strictly increasing. But then, by Lemma~\ref{lem:T_rho_{n}} below, we have
$$ T(a_{n}) - k(n)\ =\ T([E_{k(n)}(\rho_{n})]) - k(n)\ \todistr\ T^{*} (\rho). $$
On the other hand, by our assumption
$$ T(a_{n})-b_{n}\ \todistr\  Z. $$
By Theorem~\ref{mainT} and our assumption, respectively, $T^{*}(\rho)$ and $Z$ are nondegenerate random variables. Therefore the convergence of types lemma ensures the existence of $c$ (necessarily an integer) such that $Z$ and $T^{*}(\rho) + c$ have the same law.\qed
\end{proof}

\begin{Lemma}\label{lem:T_rho_{n}}
Let $(\rho_{n})_{n\in\N} \subset (1,\infty)$ be a sequence converging to $\rho>1$. Then,
$$ T([E_{n}(\rho_{n})])-n\ \todistr\ T^{*} (\rho). $$
\end{Lemma}

\begin{proof}
Fix any $k\in\Z$. For any given $\eps>0$, there is $n(\eps)$ such that $|\rho_{n}-\rho|< \eps$ for all $n\ge n(\eps)$.
Since $T([E_{n}(\rho-\eps)]) \le T([E_{n}(\rho_{n})]) \le T([E_{n}(\rho+\eps)])$ for $n\ge n(\eps)$, we obtain for $n\ge n(\eps)$ that
\begin{align*}
\Prob\{T([E_{n}(\rho+\eps)])-n\le k \}\ &\le \Prob\{T([E_{n}(\rho_{n})])-n\le k\}\ \le\ \Prob\{T([E_{n}(\rho-\eps)]) -n \le k\}.
\end{align*}
By taking the limit $n\to\infty$ and using Theorem~\ref{mainT}, we find
\begin{align*}
\Prob\{T^{*}(\rho+\eps)\le k\}\ &\le\ \liminf_{n\to\infty}\Prob\{T([E_{n}(\rho_{n})]) -n \le k \}\\
&\le \limsup_{n\to\infty}\Prob\{T([E_{n}(\rho_{n})])-n\le k \}\ \le\ \Prob\{T^{*}(\rho-\eps)\le k\}.
\end{align*}
Finally, let $\eps\downarrow 0$ and use the continuity in probability of $(T^{*}(\rho))_{\rho>1}$  (shown in Theorem~\ref{mainT}), to arrive at
$$ \lim_{n\to\infty} \Prob\{T([E_{n}(\rho_{n})]) -n \le k \}\ =\ \Prob\{T^{*}(\rho)\le k\}, $$
which is the desired conclusion.\qed
\end{proof}

\begin{proof}[of Theorem~\ref{LLN_for_T}]
The $L^{r}$-convergence follows directly from the $L^{r}$-boundedness of $T(M)-\log^{*}M$ (provided by Theorem \ref{mainT0}). To prove the almost sure convergence, fix $\eps>0$ and note that, by Markov's inequality and Theorem \ref{mainT0},
\begin{align*}
\Prob\left\{\left|\frac{T([\widetilde E_{n+1}(0)])}{\log^{*} [\widetilde E_{n+1}(0)]}-1\right|\ge \eps\right\}\ &\le\ \frac{\Erw (T([\widetilde E_{n+1}(0)])-\log^{*} [\widetilde E_{n+1}(0)])^2}{\eps^2 (\log^{*} [\widetilde E_{n+1}(0)])^2}\\
&\le\ \frac{C}{\eps^{2}(\log^{*}[\widetilde E_{n+1}(0)])^2}\\
&\le\ \frac{C}{\eps^{2}(\log^{*}\widetilde E_{n}(0))^2}\ =\ \frac{C}{\eps^{2}n^{2}}
\end{align*}
for some $C>0$. Hence, the Borel-Cantelli lemma provides us with
$$ \frac{T([\widetilde E_{n}(0)])}{\log^{*} [\widetilde E_{n}(0)]}\ \to\ 1\quad\text{a.s.} $$
Moreover, for each sufficiently large $M\in\N$, there exists $n$ such that
$$ [\widetilde{E}_{n}(0)]\le\widetilde{E}_{n}(0)\,<\,M\,\le\,\widetilde{E}_{n+1}(0)\,\le\, [\widetilde{E}_{n+2}(0)]. $$
Therefore,
$$ \frac{\log^{*}[\widetilde E_{n}(0)]}{\log^{*}M}\cdot\frac{T([\widetilde{E}_{n}(0)])}{\log^{*} [\widetilde E_{n}(0)]}\ \le\ \frac{T(M)}{\log^{*} M}\ \le\ \frac{T([\widetilde{E}_{n+2}(0)])}{\log^{*} [\widetilde E_{n+2}(0)]}\cdot\frac{\log^{*} [\widetilde E_{n+2}(0)]}{\log^{*} M}
$$
and the claim follows by a standard sandwich argument.\qed
\end{proof}

\subsection{... of Theorem \ref{mainT_{0}} and Propositions \ref{Proposition_last_rounds}, \ref{prop:T_0_ast_properties}}\label{subsec:PD-theorem-T_0}

\begin{proof}[of Proposition \ref{Proposition_last_rounds}]
We give the proof only for the one-dimensional distributions because the higher-dimensional ones are treated in a similar manner, but become notationally quite tedious. Fix $m\in\N$, $\rho>1$, $k\in\Z$ and integers $2\le l_{1}\le\ldots\le l_m$. We have
\begin{align*}
&\Prob\left\{T([E_{n}(\rho)])-n=k,N_{[E_{n}(\rho)]}^{(T([E_{n}(\rho)])-i)}=l_{i},\,i=1,\ldots,m\right\}\\
&\hspace{3cm}=\ \Prob\{N_{[E_{n}(\rho)]}^{(n+k)}=1,N_{[E_{n}(\rho)]}^{(n+k-i)}=l_{i},\,i=1,\ldots,m\}.
\end{align*}
By Theorem \ref{main2_time}, the last probability converges to
\begin{align*}
\Prob\{N^{*}_{k}(\rho)=1,N^{*}_{k-i}(\rho)=l_{i},\,i=1,\ldots,m\}
\end{align*}
as $n\to\infty$, and this limit probability in turn equals
$$ \Prob\{T^{*}(\rho)=k,N^{*}_{T^{*}(\rho)-i}(\rho)=l_{i},\,i=1,\ldots,m\} $$
because
$$ T^{*}(\rho)=\inf\{k\in\Z \colon N^{*}_{k}(\rho)=1\},\quad \rho>1, $$
which is the same as \eqref{T_explicit} by the definition of $N^{*}_k(\rho)$.\qed
\end{proof}

\begin{proof}[of Theorem \ref{mainT_{0}}]
Fix $m\in\N$. We have the decomposition
\begin{align*}
T_{0}&([E_{n}(\rho)])-n\ =\ \sum_{j=0}^{T([E_{n}(\rho)])-1}\1\left\{N_{[E_{n}(\rho)]}^{(j)}\neq N_{[E_{n}(\rho)]}^{(j+1)}\right\}-n\\
&=\ T([E_{n}(\rho)])-n\,-\,\sum_{j=0}^{T([E_{n}(\rho)])-1}\1\left\{N_{[E_{n}(\rho)]}^{(j)}=N_{[E_{n}(\rho)]}^{(j+1)}\right\}\\
&=\ \left(T([E_{n}(\rho)])-n-\sum_{j=(T([E_{n}(\rho)])-m)\vee 0}^{T([E_{n}(\rho)])-1}\1\left\{N_{[E_{n}(\rho)]}^{(j)}=N_{[E_{n}(\rho)]}^{(j+1)}\right\}\right)\\
&\hspace{.8cm}-\ \sum_{j=0}^{T([E_{n}(\rho)])-m-1}\1\left\{N_{[E_{n}(\rho)]}^{(j)}=N_{[E_{n}(\rho)]}^{(j+1)}\right\}\\
&=:\ T_{0,m}([E_{n}(\rho)])-V_{0,m}([E_{n}(\rho)])
\end{align*}
for any $n\in\N$ and note that, by Proposition \ref{Proposition_last_rounds}, as $n\to\infty$,
\begin{align*}
(T_{0,m}([E_{n}(\rho)]))_{\rho>1}\todistrfd \left(T^{*}(\rho)-\sum_{j=1}^{m}\1\left\{N^{*}_{T^{*}(\rho)-j}(\rho)=N^{*}_{T^{*}(\rho)-j+1}(\rho)\right\}\right)_{\rho>1}
\end{align*}
As $m\to\infty$, the right-hand limit converges almost surely to
\begin{equation}\label{eq:T_0_ast_def}
T_{0}^{*}(\rho)\
=\ T^{*}(\rho)-\sum_{j=-\infty}^{T^{*}(\rho)-1}\1\left\{N^{*}_{j}(\rho)=N^{*}_{j+1}(\rho)\right\}.
\end{equation}	
To verify that $T_{0}^{*}(\rho)$ is finite a.s., we need to check the a.s.\ convergence of the series
$$
\sum_{j=1}^{\infty}\1\left\{N^{*}_{-j}(\rho)=N^{*}_{-j+1}(\rho)\right\}.
$$
To this end, note that for all $j\in\N$ and $\rho>1$
\begin{align*}
N^{*}_{-j+1}(\rho)\ &=\ \#\{k\in\N:S^{*}_{k,-j+1}\le L_{-j+1}(\rho)\}\\
&=\ \#\{k\in\N:L_{1}(S^{*}_{\nu(k),-j})\le L_{-j+1}(\rho)\}\ =\ K(N^{*}_{-j}(\rho))
\end{align*}
with $(K(n))_{n\in \N}$ and $(\nu(n))_{n\in\N}$ being independent of $(N_{-j}^{*}(\rho))_{\rho>1}$. Consequently,
\begin{align*}
\Erw\sum_{j=1}^{\infty}\1\left\{N^{*}_{-j}(\rho)=N^{*}_{-j+1}(\rho)\right\}\ &=\
\Erw\sum_{j=1}^{\infty}\1\left\{K(N^{*}_{-j}(\rho))=N^{*}_{-j}(\rho)\right\}\ =\ \sum_{k=1}^{\infty}\frac{1}{k!}\sum_{j=1}^{\infty}\Prob\{N^{*}_{-j}(\rho)=k\}.
\end{align*}
Using $N_{-j}^{*}(\rho)\eqdist N^{*}(L_{-j}(\rho))=N^{*}(E_{j}(\rho))$, valid for all $j\in\N$, we obtain
\begin{align*}
\sum_{k=1}^{\infty}\frac{1}{k!}\sum_{j=1}^{\infty}\Prob\{N^{*}_{-j}(\rho)\le k\}\ &=\ \sum_{k=1}^{\infty}\frac{1}{k!}\sum_{j=1}^{\infty}\Prob\{N^{*}(E_{j}(\rho))\le k\}\\
&=\ \sum_{k=1}^{\infty}\frac{1}{k!}\sum_{j=1}^{\infty}\Prob\{S^{*}_{k+1}>E_j(\rho)\}\\
&\le\ \sum_{k=1}^{\infty}\frac{\Erw S_{k+1}^{*}}{k!}\sum_{j=1}^{\infty}\frac{1}{E_{j}(\rho)}\ <\ \infty,
\end{align*}
by Theorem \ref{SLLN/CLT_S_{n}}. This proves that $T_0^{*}(\rho)$ is finite a.s.
According to Theorem 3.2 in \cite{Billingsley:99}, it remains to prove that
\begin{equation}\label{billinglsey_double_limit}
\lim_{m\to\infty}\limsup_{n\to\infty}\Prob\{V_{0,m}([E_{n}(\rho)])\ge 1\}\ =\ 0.
\end{equation}
Use Markov's inequality to infer
\begin{align*}
\Prob\{V_{0,m}([E_{n}(\rho)])\ge 1\}\ &\le\ \Erw V_{0,m}([E_{n}(\rho)])\\
&=\ \sum_{j=0}^{\infty}\Prob\{N_{[E_{n}(\rho)]}^{(j)}=N_{[E_{n}(\rho)]}^{(j+1)},T([E_{n}(\rho)])-m > j\}\\
&=\ \sum_{j=0}^{\infty}\Prob\{N_{[E_{n}(\rho)]}^{(j)}=N_{[E_{n}(\rho)]}^{(j+1)},N_{[E_{n}(\rho)]}^{(m+j)}>1\}.
\end{align*}
Note that, with $(K^{(i)}(M))_{M\in\N}$, $i\in\N$, denoting independent copies of $(K(M))_{M\in\N}$ which are also independent of $N_{[E_{n}(\rho)]}^{(j)}$, we have
$$ N_{[E_{n}(\rho)]}^{(j+l)}\ \eqdist\ K^{(j+l\darr j+1)}(N_{[E_{n}(\rho)]}^{(j)}), $$
for $l=1,\ldots,m$. Therefore,
\begin{align*}
\sum_{j=0}^{\infty}\,&\Prob\left\{N_{[E_{n}(\rho)]}^{(j)}=K^{(j+1)}(N_{[E_{n}(\rho)]}^{(j)}),K^{(j+m \darr j+1)}(N_{[E_{n}(\rho)]}^{(j)})>1\right\}\\
&=\ \sum_{j=0}^{\infty}\sum_{k=2}^{\infty}\Prob\left\{N_{[E_{n}(\rho)]}^{(j)}=k\right\}\,\Prob\left\{K^{(j+1)}(k)=k\right\}\,\Prob\left\{K^{(j+m\darr j+2)}(k)>1\right\}\\
&=\ \sum_{k=2}^{\infty}\Prob\left\{K^{(1)}(k)=k\right\}\Prob\left\{K^{(1\uarr m-1)}(k)>1\right\}\sum_{j=0}^{\infty}\Prob\left\{N_{[E_{n}(\rho)]}^{(j)}= k\right\}.
\end{align*}
Now observe that $(N_{[E_{n}(\rho)]}^{(j)})_{j\in\N_{0}}$ is a nonincreasing Markov chain and that
$$ p(n,k)\ :=\ \sum_{j=0}^{\infty}\Prob\{N_{[E_{n}(\rho)]}^{(j)}=k\} $$
is just the mean number of visits of this chain to the state $k$. If the chain ever visits $k\ge 2$ then the number of rounds it remains in $k$ is geometrically distributed with parameter $1-1/k!$, hence
$$ p(n,k)\ \le\ \Erw\,\textrm{Geom}\left(1-\frac{1}{k!}\right)\ \le\ 2. $$
A combination of all previous estimates finally yields
\begin{align*}
\limsup_{n\to\infty}\Prob\{V_{0,m}([E_{n}(\rho)])\ge 1\}\le 2\sum_{k=2}^{\infty}\frac{1}{k!}\Prob\{K^{(1\uarr m-1)}(k)>1\}.
\end{align*}
The latter series converges uniformly in $m$ because it is dominated by $\sum_{k=2}^{\infty}\frac{1}{k!}$. Therefore, we can pass to the limit $m\to\infty$ under the sum sign, giving
\begin{align*}
\lim_{m\to\infty}\sum_{k=2}^{\infty}\,\frac{1}{k!}\Prob\{K^{(1\uarr m-1)}(k)>1\}
\ =\ \sum_{k=2}^{\infty}\frac{1}{k!}\lim_{m\to\infty}\Prob\{K^{(1\uarr m-1)}(k)>1\}\ =\ 0,
\end{align*}
which completes the proof of~\eqref{billinglsey_double_limit}.
\qed
\end{proof}

\begin{proof}[of Proposition~\ref{prop:T_0_ast_properties}]
To prove that $T^{*}_0(\rho)$ is nondegenerate, assume that $T^{*}_0(\rho)=k$ a.s. But then $T^{*}(\rho)\geq T^{*}_0(\rho) = k$, which is a contradiction to $\Prob\{T^{*}(\rho)=k-1\}>0$, see Step 4 in the proof of Theorem~\ref{mainT}.

\vspace{.1cm}
Next we fix $1<\rho_{1}<\rho_{2}$ and prove that $T^{*}_0(\rho_{1}) \leq T^{*}_0(\rho_{2})$ with a strict inequality on an event of positive probability.
 Recall from~\eqref{eq:T_0_ast_def} that 
$$
T^{*}_0(\rho) = T^{*}(\rho) -  \sum_{j=-\infty}^{T^{*}(\rho)-1}\1\left\{N^{*}_{j}(\rho)=N^{*}_{j+1}(\rho)\right\}.
$$
It is clear that $T^{*}(\rho_{1}) \leq T^{*}( \rho_{2})$.
If for some integer $j < T^{*}(\rho_{1})$ we have $N^{*}_{j}(\rho_{2}) = N^{*}_{j+1}(\rho_{2})$, then $N^{*}_{j}(\rho_{1}) = N^{*}_{j+1}(\rho_{1})$. Thus,
\begin{align*}
T^{*}_0(\rho_{2})
&=
T^{*}(\rho_{2}) -  \sum_{j=T^{*}(\rho_{1})}^{T^{*}(\rho_{2})-1}\1\left\{N^{*}_{j}(\rho_{2})=N^{*}_{j+1}(\rho_{2})\right\}
-  \sum_{j=-\infty}^{T^{*}(\rho_{1})-1}\1\left\{N^{*}_{j}(\rho_{2})=N^{*}_{j+1}(\rho_{2})\right\}\\
&\geq
T^{*}(\rho_{2}) -  \left( T^{*}(\rho_{2})-T^{*}(\rho_{1})\right)
-\sum_{j=-\infty}^{T^{*}(\rho_{1})-1}\1\left\{N^{*}_{j}(\rho_{1})=N^{*}_{j+1}(\rho_{1})\right\}\\
&=T^{*}_0(\rho_{1}).
\end{align*}
To prove the inequality to be strict with positive probability, recall $T^{*}(\rho) = \inf\{k\in\Z: E_k (S_{k,2}^{*})>\rho\}$ and note that
\begin{align*}
&\Prob
\left\{
\sum_{j=T^{*}(\rho_{1})}^{T^{*}(\rho_{2})-1}\1\left\{N^{*}_{j}(\rho_{2})=N^{*}_{j+1}(\rho_{2})\right\}
< T^{*}(\rho_{2}) - T^{*}(\rho_{1})
\right\}\\
&\quad \quad\quad \quad \geq
\Prob
\left\{
S_{0,2}^{*} < \rho_{1} < S_{0,3}^{*} < \rho_{2} < S_{0,4}^{*},  \rho_{1} < E_{1} (S_{1,2}^{*}) < \rho_{2} < E_{1} (S_{1,3}^{*}), \rho_{2} < E_{2}(S_{2,2}^{*})
\right\}\\
&\quad \quad\quad \quad =
\frac 12 \cdot \frac 13  \cdot  \frac 12 \cdot \Prob
\left\{
S_{0,2}^{*} < \rho_{1} < S_{0,3}^{*} < \rho_{2} < S_{0,4}^{*}
\right\},
\end{align*}
which is strictly positive by Prop.~\ref{prop:no_atoms}. Indeed, the condition on the $S_{k,j}^{*}$'s in the second line guarantees that $T^{*}(\rho_{1}) = 1$, $T^{*}(\rho_{2}) = 2$, and $2=N_{1}^{*}(\rho_{2}) \neq N_{2}^{*}(\rho_{2})=1$.
\qed
\end{proof}

\subsection{... of Theorems \ref{PD-theorem}, \ref{PD-theorem_tight}}\label{subsec:PD-theorem}

\begin{proof}[of Theorem \ref{PD-theorem}]
We will prove $X_{1}(n) \eqdist T_{0}(n)$, recalling that
$$ T_{0}(n)\ =\ \sum_{j=0}^{T(n)-1}\1_{\{N_{n}^{(j)}\neq N_{n}^{(j+1)}\}} $$
denotes the number of conclusive rounds in the leader election procedure using records.
It is known (see calculations on p.~2170 in \cite{Moehle:10}) that the number $J_{\theta}(n)$ of blocks after the first collision in the Poisson-Dirichlet coalescent with parameter $\theta>0$ and $n$ initial blocks has distribution
\begin{equation}\label{PD_dec_distr}
\Prob\{J_{\theta}(n)=k\}\ =\ \frac{\theta^{k}}{[\theta]_{n}-\theta^n}\stirling{n}{k},\quad 1\le k<n,\quad n\ge 2,
\end{equation}
where $\stirling{n}{k}$ denotes the unsigned Stirling number of the first kind.
For $\theta=1$, this gives
$$ \Prob\{J_{1}(n)=k\}\ =\ \frac{\stirling{n}{k}}{n!-1}\ =\ \Prob\{K(n)=k|K(n)<n\},\quad 1\le k<n,\quad n\ge 2. $$
The number of blocks in the coalescent forms a decreasing Markov chain on the set $\{1,\ldots,n\}$ which starts at $n$ and is eventually absorbed at state $1$.  The transitions of this chain can be described as follows. Generate $n$ iid variables with continuous distribution and count the number $K(n)$ of records among them. If there are exactly $n$ records, repeat the procedure until at some point one gets a number $J_{1}(n) < n$ records. The Markov chain jumps from state $n$ to state $J_{1}(n)$, and the procedure is repeated independently from the past. After a random number $X_{1}(n)$ of downward jumps, the Markov chain is absorbed at state $1$.
From this description, it follows that
\begin{equation}\label{recursion_X}
X_{1}(1)\ =\ 0\quad\text{and}\quad X_{1}(n)\ \eqdist\ \1\{K(n)<n\}+\wh{X}_{1}(K(n))\quad\text{for }n\ge 2,
\end{equation}
where again $(\wh{X}_{1}(n))_{n\in\N}\eqdist (X_{1}(n))_{n\in\N}$ and $(\wh{X}_{1}(n))_{n\in\N}$ is independent of $K(n)$.

\vspace{.1cm}
Let us show that the  same recursion holds for $(T_0(n))_{n\in\N}$. Clearly, $T_{0}(1)=0$. Conditioning on the number of players after the first round, which is distributed as $K(n)$, we see that for $n\ge 2$ the following equality in distribution holds:
\begin{equation}\label{recursion_T_{0}}
T_{0}(n)\ \eqdist\ \1\{K(n)<n\}+\wh{T}_{0}(K(n)),
\end{equation}
where $\wh{T}_{0}$ has the obvious meaning. Therefore, by induction over $n\in\N$, we obtain $T_{0}(n)\eqdist X_{1}(n)$ for all $n$  and the result finally follows from Corollary \ref{mainT_{0}_tilde}.
The properties of the subsequential limits $\widetilde T_0^{*}(\rho) = T_0^{*}(f(\rho))$, $\rho\in\R$, follow from the corresponding properties of $T_0^{*}(\rho)$, $\rho>1$, stated in Proposition~\ref{prop:T_0_ast_properties} combined with the strict monotonicity of $f$.
\qed
\end{proof}

\begin{proof}[of Theorem \ref{PD-theorem_tight}]
Since $X_{1}(M)$ and $T_0(M)$ have the same distribution, it suffices to show that the sequence $(T_0(M) - \log^{*} M)_{M\in\N}$ is $L^{r}$-bounded for every $r>1$. Recall that $T_{0}(M) = T(M) - T_{1}(M)$ with
$$
T_{1}(M) \ =\  \sum_{j=0}^{T(M)-1} \1\{N_{M}^{(j)} = N_{M}^{(j+1)}\}
$$
denoting the number of inconclusive rounds. Noting that $(T(M) - \log^{*} M)_{M\in\N}$ is $L^{r}$-bounded by Theorem~\ref{mainT}, our task reduces to showing that $(T_{1}(M))_{M\in\N}$ is $L^{r}$-bounded. This is accomplished by proving that for every $M\in\N$ the random variable $T_{1}(M)$ is stochastically dominated by $S := \sum_{k=2}^{\infty} (G_k-1)$, where $G_{2}, G_{3},\ldots$ are independent and $G_{k}$ is geometric with success probability $1-1/k!$. Note that $(N^{(j)}_M)_{j\in\N_0}$ forms a nonincreasing Markov chain on $\{1,\ldots,M\}$ starting in $M$ and absorbed in state 1 at time $T(M)$. Let $G_{2}, G_3,\ldots$ be as above and independent of $(N^{(j)}_M)_{j\in\N_0}$. Denoting by $I\subset \{1,\ldots,M\}$ the random set of states visited by this Markov chain, we have
$$
T_{1}(M)\ \od \ \sum_{k\in I\setminus\{1\}}(G_{k}-1)\ \le\ \sum_{k=2}^{\infty} (G_{k}-1)\ =\ S,
$$
thus proving that $T_{1}(M)$ is stochastically dominated by $S$. The $r$-th moment of $S$ is finite because, by using Minkowski's inequality for the $L^{r}$-norm $\|\cdot\|_r$,
$$
\| S \|_r
\ \leq\
\sum_{k=2}^{\infty} \|G_k-1\|_r
\ =\
\sum_{k=2}^{\infty} \left(\sum_{l=1}^{\infty} \frac{l^{r}}{k!^l}\left(1-\frac{1}{k!}\right)\right)^{1/r}
\ \leq \
\sum_{k=2}^{\infty} \frac 1 {(k!)^{1/r}}\left(\sum_{l=1}^{\infty} \frac{l^{r}}{2^{l-1}}\right)^{1/r}
\ < \ \infty.
$$


It follows that the sequence $(X_{1}(M) - \log^* M)_{M\in\N}$ is $L^{r}$-bounded for every $r>1$ and hence tight.  However, this sequence does not converge in distribution because the subsequential limits in  Theorem~\ref{mainT_{0}} have pairwise distinct distributions by Prop.~\ref{prop:T_0_ast_properties}.
\qed
\end{proof}

\subsection{... of \eqref{N_n_fdd_tilde} of Corollary \ref{main2p}}\label{subsec:main2p}
All results in Subsection \ref{sec::st_tetration} are proved in a similar manner and rely on the simple observation that
$$ \lim_{n\to\infty}L_{n+j}(\widetilde{E}_{n}(\rho))=L_{j}(f(\rho)),\quad \rho\in\R, $$
and the fact that, by Prop.~\ref{prop:no_atoms}, $S_{j}^{*}$ has a continuous distribution function for every $j\ge 2$. Therefore, we confine ourselves to a complete proof of \eqref{N_n_fdd_tilde} of Corollary \ref{main2p}.

\begin{proof}[of Corollary \ref{main2p}, Eq.~\eqref{N_n_fdd_tilde}]
Fixing $m\in\N$, $\rho_{1}<\ldots<\rho_m$, and $k_{1},\ldots,k_m\in\N$, we infer by another use of the duality relation \eqref{duality_S_N}
\begin{align*}
\Prob\{N_{[\widetilde E_{n}(\rho_{i})]}^{(n)} \ge k_{i},\,i=1,\ldots,m\}\ &=\
\Prob\{S_{k_{i}}^{(n)} \le [\widetilde E_{n}(\rho_{i})],\,i=1,\ldots,m\}\\
&=\ \Prob\{S_{k_{i}}^{(n)} \le \widetilde E_{n}(\rho_{i}),\,i=1,\ldots,m\}\\
&=\ \Prob\{L_{n}(S_{k_{i}}^{(n)}) \le L_{n}(\widetilde E_{n}(\rho_{i})),\,i=1,\ldots,m\}.
\end{align*}
By the definition of $f$ (see \eqref{f_def}) and Prop.~\ref{prop:no_atoms}, we have
$$ \lim_{n\to\infty} L_{n}(\widetilde E_{n}(\rho_{i}))=f(\rho_{i})
\quad\text{and}\quad\Prob\{{S}^{*}_{k_{i}}=f(\rho_{i})\}=0
$$
for $i=1,\ldots,m$. Applying Theorem~\ref{main1}, we obtain
\begin{align*}
\lim_{n\to\infty}\Prob\{N_{[\widetilde E_{n}(\rho_{i})]}^{(n)} \ge k_{i},\,i=1,\ldots,m\}\ &=\
\Prob\{{S}_{k_{i}}^{*}\le f(\rho_{i}),\,i=1,\ldots,m\}\\
&=\ \Prob\{{N}^{*}(f(\rho_{i}))\ge k_{i},\,i=1,\ldots,m\}\\
&=\ \Prob\{\widetilde{N}^{*}(\rho_{i})\ge k_{i},\,i=1,\ldots,m\},
\end{align*}
and this shows \eqref{N_n_fdd_tilde}.\qed
\end{proof}

\section{Appendix}
\subsection{Asymptotics of the record times $\nu(n)$}
Recall that $\nu(n)$ denotes the $n$th record time in a sequence of iid observations with continuous distribution.
R\'enyi~\cite{Renyi:62} has shown that $\log \nu(n)$ satisfies a law of large numbers and a central limit theorem, namely
\begin{equation}\label{eq:lln_clt_{n}u_{n}}
\frac{\log \nu(n)}{n} \to 1\ ~\text{a.s.}\quad\text{and}\quad
\frac{\log \nu(n)-n}{\sqrt{n}}\ \dod\ \text{N}(0,1),
\end{equation}
where $\text{N}(0,1)$ denotes the standard normal law. Some additional properties are collected in the subsequent proposition.

\begin{Prop}\label{nu_{n}_prop}
For every $r>0$, we have
\begin{align*}
&\sup_{n\in\N} \Erw\left(\frac{\log \nu(n)}n\right)^{r}<\infty,
\quad\sup_{n\in\N} \Erw\left|\frac{\log \nu(n)-n}{\sqrt{n}}\right|^{r}<\infty,\\
&\sup_{n\in\N} \Erw\left(\frac{n}{1+\log \nu(n)}\right)^{r} <\infty.
\end{align*}
\end{Prop}
\begin{proof}

The first two assertions were proved by Gut~\cite[Thm.~2]{Gut:90}. To prove the third one, we use Williams' representation for record times~\cite[page 60]{Nevzorov:01}. Let $U_{1},U_{2},\ldots$ be independent random variables distributed uniformly on the interval $[0,1]$. Define $R(1)=1$ and $R(n+1) = \lceil R(n)/U_{n}\rceil$, for $n\in\N$. The representation  states that the random sequences  $(\nu(n))_{n\in\N}$ and $(R(n))_{n\in\N}$ have the same distribution. We have $R(n+1)\ge R(n)/U_{n}$ and hence,
\begin{equation}\label{eq:R_{n}_williams_estimate}
R(n)\ \ge\ \frac{1}{U_{1}\cdot\ldots\cdot U_{n-1}}, \quad n\in\N.
\end{equation}
Note that the variables $\cE_k:=-\log U_k$ are standard exponential. Fix any $r>0$ and let $n>r$. By~\eqref{eq:R_{n}_williams_estimate}, we have
\begin{align*}
\Erw\left(\frac{n}{1+\log\nu(n)}\right)^{r}\ &=\ \Erw\left(\frac{n}{1+\log R(n)}\right)^{r}\
\le\ \Erw\left(\frac{n}{1+\sum_{k=1}^{n-1}\cE_k}\right)^{r}\\
&\le\ \Erw\left(\prod_{k=1}^{n-1} \cE_k^{-1/n}\right)^{r}\ =\ \left(\Erw(\cE_{1}^{-r/n})\right)^{n-1},
\end{align*}
where the inequality between arithmetic and geometric means has been utilized. Furthermore, with $\Gamma$ denoting the Gamma function and $\gamma$ the Euler-Mascheroni constant,
$$ \big(\Erw(\cE_{1}^{-r/n})\big)^{n-1}\ =\ \left(\int_{0}^{\infty}y^{-r/n}e^{-y}{\rm d}y\right)^{n-1}\ =\ \Gamma^{n-1}\left(1-\frac rn\right) \to\ e^{r\gamma} $$
as $n\to\infty$.\qed
\end{proof}

\begin{Lemma}\label{expectation_lemma}
For all $n\ge 2$, we have
$$ \Erw\left(\frac{1}{\nu(n)}\right) \le \frac{3}{4n}. $$
\end{Lemma}

\begin{proof}
If $n=2$, we have $\Prob\{\nu(2)=k\} = \frac{1}{(k-1)k}$ for $k\ge 2$; see~\cite[page 56]{Nevzorov:01}. Consequently,
\begin{align*}
\Erw\left(\frac{1}{\nu(2)}\right)\ &=\ \sum_{k=2}^{\infty}\frac{1}{k^2(k-1)}\ =\ \sum_{k=2}^{\infty}\frac{1}{k(k-1)}\,-\,\sum_{k=2}^{\infty}\frac{1}{k^2}
\ =\ 2-\frac{\pi^2}{6}\ <\ \frac{3}{4} \cdot \frac{1}{2}.
\end{align*}
For $n\ge 3$, we make another use of Williams' representation. Applying~\eqref{eq:R_{n}_williams_estimate}, we obtain
$$ \Erw\left(\frac1 {\nu(n)}\right)\ =\ \Erw\left(\frac1 {R(n)}\right)\ \le\ \Erw(U_{1}\cdot\ldots \cdot U_{n-1})\ =\ 2^{-(n-1)} $$
for all $n\in\N$, and since $2^{-(n-1)}\le\frac 3{4n}$ for $n\ge 3$, as one can easily verify by induction, the lemma is proved.\qed
\end{proof}

\subsection{Conjugating $E_{n}$ and $\widetilde E_{n}$}

In this section we prove the existence of the function $f$ defined in Subsection \ref{sec::st_tetration}.

\begin{Lemma}\label{f_existence}
Let $f_{n}(z):=L_{n}(\widetilde{E}_{n}(z))$. Then $f(z):=\lim_{n\to\infty}f_{n}(z)$ exists for every $z\in\R$ and satisfies $f(z)>1$. Moreover, $f$ is continuous and strictly increasing.
\end{Lemma}
\begin{proof}
Note first that $f_{1}(z)<f_{2}(z)<\ldots$ for every $z\in\R$. Indeed,
\begin{align*}
f_{n}(z)\ &=\ L_{n-1}(1+\log\widetilde{E}_{n}(z))\ >\ L_{n-1}(\log\widetilde{E}_{n}(z))\\
&=\ L_{n-1}(\widetilde{E}_{n-1}(z))\ =\ f_{n-1}(z).
\end{align*}
Assume for a moment that $z\ge 1$ and let us prove by induction that $f_{n}(z)\le e+z$ for all $n\in\N$. For $n=1$ this is true, since $f_{1}(z)=1+z\le e+z$. For $n\ge 2$, we have
\begin{align*}
f_{n}(z)\,&=\,1+\log f_{n-1}(e^z)\,\le\,1+\log (e+e^z)\,\le\,1+\log e+\log e^z\,\le\,e+z,
\end{align*}
where the subadditivity of $z\mapsto \log z$ for $z\ge e$ has been utilized. Since $f_{n}(z)$ is nondecreasing in $z$, we arrive at the uniform bound
\begin{equation}\label{f_upper_bound}
f_{n}(z)\le (e+z)\vee (e+1)\quad\text{for all }z\in\R\text{ and }n\in\N.
\end{equation}
Hence for every fixed $z\in\R$ the sequence $(f_{n}(z))_{n\in\N}$ is nondecreasing and bounded and therefore converges to some limit. The trivial lower bound is
\begin{equation}\label{f_lower_bound}
f_{n}(z)\ge 1\vee z,\quad z\in\R,\quad n\in\N.
\end{equation}
It remains to show that $f$ is continuous and strictly increasing. In view of the functional equation \eqref{f_equation}, it is enough to check this on $[0,\infty)$. For each $n\in\N$, the function $f_{n}$ is differentiable on $\R$ and, by \eqref{f_upper_bound},
\begin{align*}
f_{n}'(z)\ &=\ \frac{f_{n-1}'(e^z)}{f_{n-1}(e^z)}e^z\ \ge\ f_{n-1}'(e^z)\frac{e^z}{e+e^z}\
\ge\ f_{n-2}'(e^{e^z})\frac{e^z}{e+e^z}\cdot\frac{e^{e^z}}{e+e^{e^z}}\\
&\ge\ldots\ge\  f_{1}'(\widetilde{E}_{n-1}(z))\prod_{k=1}^{n-1}\frac{\widetilde{E}_k(z)}{e+\widetilde{E}_k(z)}
= \prod_{k=1}^{n-1}\frac{\widetilde{E}_k(z)}{e+\widetilde{E}_k(z)}.
\end{align*}
As a consequence,
$$
C:= \lim_{n\to\infty}\Big(\inf_{z\ge 0}f_{n}'(z)\Big)\ \ge\ \prod_{k=1}^{\infty}\frac{\widetilde{E}_k(0)}{e+\widetilde{E}_k(0)}\ >\ 0. $$
The function $f$ is obviously nondecreasing, and for arbitrary $z_{1}>z_{2}\ge 0$, we obtain
\begin{align*}
f(z_{1})-f(z_{2})\ &=\ \lim_{n\to\infty}(f_{n}(z_{1})-f_{n}(z_{2}))\ =\ (z_{1}-z_{2}) \lim_{n\to\infty}f'_{n}(\theta_{n})
\geq C(z_{1}-z_{2})\ >\ 0
\end{align*}
and thus the strict monotonicity of $f$. Analogously, to show that $f(z)$ is continuous for $z\ge 0$ and then everywhere, it is enough to check that
$$ \sup_{z\ge 0}f_{n}'(z)\ \le\ 1, \quad n\in\N,$$
but this is obvious in view of
$$ f_{n}'(z)\ =\ \frac{f_{n-1}'(e^z)}{f_{n-1}(e^z)}e^z\ \le\ f_{n-1}'(e^z)\ \le\ldots \le\ f_{1}'(\widetilde{E}_{n-1}(z))\ =\ 1, $$
which holds by \eqref{f_lower_bound}. The proof is complete.\qed
\end{proof}

\subsection{Simulation methods}\label{sec:simulation}
The numerical simulation of the leader-election procedure is a non-trivial task because the straightforward approach based on the definition leads very quickly to both, overflows and extremely large computational times. Let us explain briefly the method we used to generate samples of $S_{j}^{*}$. In order to simulate the $n$-th record time $\nu(n)$, we used Williams' representation~\cite[page 60]{Nevzorov:01} if $n\le 15$, and the central limit approximation (see~\eqref{eq:lln_clt_{n}u_{n}}) if $n>15$. In this way, one can generate the backward iterates $\nu^{(k\darr 1)}(j) = \nu^{(k)} \circ \ldots \circ \nu^{(1)} (j)$. Although this leads to an overflow rather quickly, the simulations show that the rate of convergence of $\mathcal{X}_{k,j}:=L_{k}(\nu^{(k\darr 1)}(j))$ to its limit $\mathcal{X}_{j}^{*}$ is high; see~\eqref{eq:sum_expect_diff_X_{n}_j} for a theoretical estimate.  This allowed us to obtain approximate realizations of $\mathcal{X}_{j}^{*}$ which has the same distribution as $S_j^{*}$.  Figure~\ref{fig:distr_funct_S} shows the empirical distribution functions  of $S_{2}^{*}, S_{3}^{*}, S_{4}^{*}$ based on samples of size $>5000$. To obtain an estimate for the distribution of $T^{*}(\rho)$, see Figure~\ref{fig:distr_T_rho}, we used the identity
\begin{align*}
\Prob\{T^{*}(\rho)=k\}\ &=\ \Prob\{T^{*}(\rho)\le k\} -\Prob\{T^{*}(\rho)\le k-1\}\\
&=\ \Prob\{L_{k}(\rho)\le S_{2}^{*}\le L_{k-1}(\rho)\},
\end{align*}
for $k\in\Z$, which follows from~\eqref{eq:S_{2}_ast_T_rho_connection}.

\acknowledgement{The authors would like to express their sincere gratitude to two anonymous referees for numerous suggestions and inspiring comments on the first version of the manuscript. The research of Gerold Alsmeyer was supported by the Deutsche Forschungsgemeinschaft (SFB 878), the research of Alexander Marynych by the Alexander von Humboldt Foundation.

\bibliographystyle{abbrv}
\bibliography{StoPro}

\end{document}